\documentclass[a4paper,10pt,twoside]{article}

\usepackage[utf8]{inputenc}
\usepackage[T1]{fontenc}
\usepackage[francais]{babel}
\usepackage{amsmath, amssymb, amsthm}
\usepackage{amstext, amsfonts, a4}
\usepackage{enumerate}
\usepackage{latexsym}
\usepackage{mathrsfs}
\usepackage[all]{xy}

\DeclareMathOperator{\Ass}{Ass}
\DeclareMathOperator{\Br}{Br}
\DeclareMathOperator{\cyc}{cyc}

\DeclareMathOperator{\Ker}{Ker}

\DeclareMathOperator{\Spec}{Spec}

\DeclareMathOperator{\Hom}{Hom}

\DeclareMathOperator{\red}{red}
\DeclareMathOperator{\sh}{hs}

\DeclareMathOperator{\fppf}{fppf}
\DeclareMathOperator{\fpqc}{fpqc}

\DeclareMathOperator{\Pic}{Pic}

\DeclareMathOperator{\ltf}{ltf}

\DeclareMathOperator{\nor}{nor}
\DeclareMathOperator{\prof}{prof}

\DeclareMathOperator{\reg}{reg}

\DeclareMathOperator{\cO}{\mathcal{O}}
\DeclareMathOperator{\cL}{\mathcal{L}}
\DeclareMathOperator{\cV}{\mathcal{V}}
\DeclareMathOperator{\cI}{\mathcal{I}}

\DeclareMathOperator{\cN}{\mathcal{N}}

\DeclareMathOperator{\cP}{\mathcal{P}}
\DeclareMathOperator{\cM}{\mathcal{M}}

\DeclareMathOperator{\bbG}{\mathbb{G}}

\DeclareMathOperator{\bbZ}{\mathbb{Z}}

\DeclareMathOperator{\fm}{\mathfrak{m}}

\DeclareMathOperator{\ra}{\rightarrow}

\DeclareMathOperator{\NS}{NS}

\def\div{\mathrm{div}}

\newtheorem{Th}{Théorème}[section]
\newtheorem{Cor}[Th]{Corollaire}
\newtheorem{Prop}[Th]{Proposition}
\newtheorem{Lem}[Th]{Lemme}
\newtheorem{Rem}[Th]{Remarque}

\newtheorem{Def}[Th]{Définition}

\title{\textbf{Modèles semi-factoriels et modèles de Néron}}
\author{Cédric Pépin}

\date{}

\begin{document}

\maketitle

\begin{abstract}
Soit $S$ le spectre d'un anneau de valuation discrète de corps de fonctions $K$. Soit $X$ un schéma sur $S$. On dira que $X$ est \emph{semi-factoriel sur $S$} si tout faisceau inversible sur la fibre générique $X_{K}$ se prolonge en un faisceau inversible sur $X$. On montre ici que tout schéma propre géométriquement normal sur $K$ possède un modèle propre, plat, normal et semi-factoriel sur $S$. On construit également des compactifications semi-factorielles de $S$-schémas réguliers, tels que les modèles de Néron des variétés abéliennes. 

La propriété de semi-factorialité pour un schéma $X/S$ correspond à la propriété de Néron de son foncteur de Picard. En particulier, on peut retrouver le modèle de Néron de la variété de Picard $\Pic_{X_K/K,\red}^0$ de $X_K$ à partir du foncteur de Picard $\Pic_{X/S}$, comme dans le cas connu des courbes. On en tire des conséquences sur l'équivalence algébrique relative sur le $S$-schéma $X$.
\end{abstract}

\tableofcontents

\section{Introduction} \label{section0}
Soit $S$ un trait, c'est-à-dire le spectre d'un anneau de valuation discrète. On note $K$ son corps de fonctions. 

\begin{Def} \label{tf}
Soit $X\ra S$ un $S$-schéma. On dit que $X$ est \emph{semi-factoriel sur $S$} si l'homomorphisme de restriction $$\Pic(X)\ra\Pic(X_{K})$$
est surjectif.  
\end{Def}
Un $S$-schéma $X/S$ est donc semi-factoriel si pour tout $\cO_{X_{K}}$-module inversible $\cL_{K}$, il existe un $\cO_{X}$-module inversible $\cL$ tel que $\cL\otimes K$ soit isomorphe à $\cL_{K}$. Cette définition est surtout pertinente lorsque le schéma $X$ est \emph{localement noethérien et normal}, parce qu'on dispose alors d'une bonne interprétation des faisceaux inversibles sur $X$ en termes de cycles: le groupe de Picard de $X$ \emph{s'injecte} dans celui des classes de cycles $1$-codimensionnels (\cite{EGA IV}${}_4$ 21.6.10). La terminologie << semi-factoriel >> provient du cas où le schéma $X$ est localement factoriel : il est alors automatiquement semi-factoriel sur $S$ (\emph{loc. cit.}). 

En l'abscence actuelle d'un théorème de désingularisation général, on ne peut pas assurer qu'un schéma propre et lisse sur $K$ possède un modèle popre et plat sur $S$, qui soit un schéma \emph{régulier} (par << modèle >>, on entend << $S$-schéma de fibre générique isomorphe à $X_K$ >>).

Dans la section \ref{section1}, on montre néanmoins que tout schéma $X_K$ qui est propre géométriquement normal sur $K$ possède un modèle $X$ propre, plat, normal et \emph{semi-factoriel} sur $S$. Plus généralement, si $T$ est un trait limite projective filtrante de traits étales sur $S$, on construit des modèles $X/S$ comme ci-dessus qui commutent au changement de base $T\ra S$, i.e. $X\times_S T$ est semi-factoriel sur $T$. Lorsque $T$ est un hensélisé strict $S^{\sh}$ de $S$, la construction fournit des modèles qui commutent à des changements de trait plus généraux que $S^{\sh}\ra S$, que l'on a qualifiés de \emph{permis} (définition \ref{permis} et théorème \ref{universel}). Dans la sous-section \ref{scsf}, on part d'un schéma $X$ de type fini, séparé et plat sur $S$, qui est \emph{régulier}. On construit alors des compactifications normales $\overline{X}$ de $X$ telles que la flèche de restriction $\Pic(\overline{X})\ra\Pic(X)$ soit surjective, et qui sont compatibles aux changements de trait du type précédent (corollaire \ref{vauniv}). Cela s'applique par exemple au cas où $X/S$ est le modèle de Néron d'une variété abélienne (corollaire \ref{mn}). Dans la sous-section \ref{global}, on montre une variante des énoncés de semi-factorialité sur un schéma de Dedekind global.

Pour construire des modèles semi-factoriels, on procède en deux étapes. On considère d'abord un morphisme lisse de type fini $f:Y\ra B$ entre schémas noethériens, et un $\cO_{Y}$-module cohérent $\cM$ qui est inversible au-dessus d'un ouvert schématiquement dense $U\subset B$. En utilisant les techniques de platification de Raynaud-Gruson \cite{GR}, on montre alors qu'après un éclatement bien choisi \emph{de la base $B$} centré en dehors de $U$, le faisceau $\cM|_{f^{-1}(U)}$ possède un prolongement inversible sur $X$ (théorème \ref{prolongement}). On applique ensuite cet énoncé à une situation universelle $f:=p_1:X\times_S\Lambda\ra X$, où $X$ est un modèle propre et plat de $X_K$. L'espace de paramètres $\Lambda$ est construit à partir du schéma de Picard rigidifié de $X_K$ et du modèle de Néron de sa composante neutre réduite. Un effort particulier a été fait pour que la construction fonctionne même en l'abscence de point $K$-rationnel sur $X_K$.

Dans la section \ref{section2}, on considère un $K$-schéma propre géométriquement normal et géométriquement connexe $X_K$, ainsi qu'un modèle propre et plat $X$ de $X_K$ sur $S$, normal, et semi-factoriel sur $S^{\sh}$. Lorsque $\Pic_{X_K/K}(K^{\sh})=\Pic(X_{K^{\sh}})$ (par exemple si $X_K(K^{\sh})\neq\emptyset$ ou si le corps résiduel de $\cO(S)$ est parfait), la semi-factorialité de $X_{S^{\sh}}/S^{\sh}$ signifie précisément que le foncteur $\Pic_{X/S}$ vérifie la propriété de Néron d'extension des points étales. On peut alors retrouver le modèle de Néron $A$ de la variété de Picard $A_K:=\Pic_{X_K/K,\red}^0$ à partir du foncteur $\Pic_{X/S}$ (théorèmes \ref{Néron} et \ref{Néronbis}). En dimension relative $1$, où l'on peut utiliser des modèles \emph{réguliers} de $X_K$, il s'agit du théorème \cite{R} 8.1.4 de Raynaud. Une fois que l'on dispose des modèles semi-factoriels, la méthode se généralise tout de suite en dimension supérieure, à ceci près que le foncteur de Picard n'est plus formellement lisse en général. Il faut donc ajouter une étape de \emph{lissification des groupes}, au sens de \cite{BLR} page 174.

En analysant le lien entre les composantes neutres de $\Pic_{X/S}$ et $A$ dans la situation précédente, on obtient des informations sur l'équivalence algébrique sur le $S$-schéma $X$. Plus précisément, supposons $\cO(S)$ complet à corps résiduel algébriquement clos. Soit $n$ l'exposant du groupe des composantes connexes de la fibre spéciale de $A/S$. Si $\cL_K$ est un faisceau inversible algébriquement équivalent à zéro sur $X_K$, alors $\cL_{K}^{\otimes n}$ peut se prolonger en un faisceau inversible algébriquement équivalent à zéro sur $X$ (corollaire \ref{n}). Cette information sera utilisée dans \cite{P} pour étudier le symbole de Néron sur $X_K$ (cf. \cite{N}), dans un cadre de géométrie relative inconditionnel.
\paragraph{Remerciements.}
Je tiens à remercier Michel Raynaud pour m'avoir confié la construction des modèles semi-factoriels, et pour m'avoir fait découvrir les \emph{\'Eléments de géométrie algébrique}. Je remercie également Qing Liu pour sa relecture attentive et ses remarques utiles. 
 
\section{Construction des modèles semi-factoriels} \label{section1}

\subsection{Un théorème pour les morphismes lisses} \label{tml}

Rappelons (\cite{GR} 5.1.3) que si $U$ est un ouvert d'un schéma localement noethérien $S$, un éclatement \emph{$U$-admissible} de $S$ est un éclatement $S'\ra S$ centré en dehors de $U$. C'est en particulier un isomorphisme au-dessus de $U$. Nous allons travailler avec un ouvert $U$ \emph{schématiquement dense} dans $S$ (\cite{EGA IV}${}_3$ 11.10.2). Lorsque $S$ est localement noethérien, cela signifie que $U$ contient tous les \emph{points associés} à $S$ (\cite{EGA IV}${}_2$ 3.1.8). Si $S'\ra S$ est un éclatement $U$-admissible avec $U$ schématiquement dense dans $S$, l'image réciproque de $U$ dans $S'$ est schématiquement dense dans $S'$ (\cite{GR} 5.1.2 (iii)).

Par ailleurs, étant donné un morphisme de schémas $f:X\ra S$ et un sous-schéma ouvert $U\subseteq S$, on notera $X_U:=X\times_S U$.

\begin{Th} \label{prolongement}
Soient $S$ un schéma noethérien, $U$ un ouvert schématiquement dense dans $S$, $f:X\ra S$ un $S$-schéma lisse de type fini, $\cM$ un $\cO_{X}$-module cohérent. Supposons que le $\cO_{X_U}$-module $\cM|_{X_U}$ soit inversible. Alors il existe un carré cartésien
\begin{displaymath}
\xymatrix{
X\ar[d]^f &\ar[l] X' \ar[d]^{f'} \\
S & \ar[l] S'
}
\end{displaymath}
où $S'\ra S$ est un éclatement $U$-admissible, et un $\cO_{X'}$-module \emph{inversible} $\widetilde{\cM}$ vérifiant la condition suivante : notant avec un $'$ les pullbacks par $S'\ra S$, on a 
$$\widetilde{\cM}|_{X_{U'}'}=\cM'|_{X_{U'}'}.$$
\end{Th}

Commençons par un critère pour qu'un module cohérent sur un $S$-schéma lisse soit inversible.

\begin{Prop} \label{clé}
Soient $S$ un schéma localement noethérien, $U$ un ouvert dense de $S$, $f:X\ra S$ un $S$-schéma \emph{lisse}, $\cM$ un $\cO_X$-module cohérent. On suppose que $\cM$ est inversible sur  $X_U$, $S$-plat, sans composante immergée sur les fibres de $X/S$ (i.e. $\cM\otimes k(s)$ est sans composante immergée pour tout $s\in S$). Alors $\cM$ est inversible.
\end{Prop}

\begin{proof}
Il suffit de montrer que pour tout $s\in S$, le module $\cM\otimes k(s)$ est inversible sur $f^{-1}(s)$. En effet, soient dans ce cas $x\in X$ et $s=f(x)$. Comme $\cM\otimes k(s)$ est inversible, il existe, quitte à remplacer $X$ par un voisinage ouvert de $x$, un homomorphisme $\alpha:\cO_X\ra \cM$ tel que $\alpha_x\otimes k(s)$ soit bijectif. Par Nakayama, le conoyau de $\alpha_x$ est nul. Notant $\cN$ le noyau de $\alpha$, on a donc une suite exacte de $\cO_{X,x}$-modules
$$0\ra \cN_x\ra \cO_{X,x} \xrightarrow{\alpha_x} \cM_x \ra 0.$$
Alors, $\cM_x$ étant $\cO_{S,s}$-plat, on a $\cN_x\otimes k(s)=0$ puis $\cN_x=0$ à nouveau par Nakayama. Ainsi, $\alpha$ est un isomorphisme au voisinage de $x$. 

Prouvons donc que $\cM$ est inversible sur les fibres de $f$. Notons d'abord que $\cM$ est inversible sur un ouvert \emph{$S$-dense} $W$ de $X$, c'est-à-dire dense dans chaque fibre de $f$ : le critère de platitude par fibres (\cite{EGA IV}${}_3$ 11.3.10) et le théorème de platitude générique (\cite{EGA IV}${}_2$ 6.9.1) assurent que $\cM$ est localement libre sur un ouvert $S$-dense de $X$, et alors $\cM$ est inversible sur cet ouvert puisqu'il l'est déjà sur l'ouvert dense $X_U$.

Soient $s\in S-U$. Choisissons une générisation $\eta$ de $s$ dans $U$. D'après \cite{EGA II} 7.1.9, il existe un trait $T$ et un morphisme $T\ra S$ envoyant le point fermé $t$ de $T$ sur $s$ et son point générique sur $\eta$. Le schéma $X_T:=X\times_S T$
est alors régulier. Notons $\cM_T$ le pull-back de $\cM$ sur $X_T$. Il est de profondeur au moins $2$ en tout point $z$ de $X_T-W_T$ : cela résulte de \cite{EGA IV}${}_2$ 6.3.1, puisque le point $z$ n'est pas associé à $\cM_T\otimes k(t)$ et que $\prof(\cO_{T,t})=1$. On en déduit que le module $\cM_T$ sur le schéma régulier $X_T$ est \emph{réflexif} (cf. \cite{S}, \S  3, Corollaires 2 et 3). En particulier, il est sans torsion, et étant de rang $1$, c'est un idéal fractionnaire de $\cO_{X_T}$. Mais un idéal fractionnaire réflexif est divisoriel (\cite{B} Chap. 7, \S  4, n${}^\circ$2, Exemple 2), et par conséquent $\cM_T$ est inversible sur le schéma localement factoriel $X_T$. Sa restriction à la fibre fermée de $X_T/T$ est inversible, donc $\cM\otimes k(s)$ est inversible. 
\end{proof}
Pour réaliser une situation dans laquelle la proposition \ref{clé} s'applique, nous avons besoin d'un procédé d'élimination des composantes immergées (comparer à \cite{F IV} 2.3).

\begin{Th} \label{composantes}
Soient $S$ un schéma noethérien, $U$ un ouvert schématiquement dense dans $S$, $f:X\ra S$ un $S$-schéma de type fini et plat à fibres géométriquement réduites, $\cM$ un $\cO_{X}$-module cohérent. Supposons que $\cM|_{X_U}$ soit localement libre de rang $\geq 1$, et qu'il existe un ouvert $S$-dense $V$ de $X$ tel que $\cM|_V$ soit localement libre.
Alors il existe un carré cartésien
\begin{displaymath}
\xymatrix{
X\ar[d]^f &\ar[l] X' \ar[d]^{f'} \\
S & \ar[l] S'
}
\end{displaymath}
où $S'\ra S$ est un éclatement $U$-admissible, et un $\cO_{X'}$-module cohérent $\widetilde{\cM}$ vérifiant les conditions suivantes : notant avec un $'$ les pullbacks par $S'\ra S$,
\begin{itemize}
\item $\widetilde{\cM}$ coïncide avec $\cM'$ sur $X_{U'}'\cup V'$ ; 
\item $\widetilde{\cM}$ est $S'$-plat ;
\item $\widetilde{\cM}$ est sans composante immergée sur les fibres de $X'/S'$.
\end{itemize}
En particulier, le $\cO_{X'}$-module $\widetilde{\cM}$ est $Z'$-clos pour $Z':= X'-X_{U'}'\cup V'$ (\cite{EGA IV}${}_2$ 5.9).
\end{Th}
La dernière assertion résulte du lemme suivant. 

\begin{Lem} \label{profondeur}
Soient $S$ et $X$ des schémas localement noethériens, $U$ un ouvert schématiquement dense dans $S$,  $f:X\ra S$ un $S$-schéma à fibres réduites, $W$ un ouvert $S$-dense de $X$ contenant $X_U$, $\cM$ un $\cO_{X}$-module cohérent et $S$-plat. Notant $\Ass$ pour << points associés >>, supposons la condition suivante vérifiée : $\Ass(\cM\otimes k(s))$ est contenu dans $\Ass(X_s)$ pour tout $s\in S$ (par exemple, $\cM$ sans composante immergée sur les fibres de $X/S$ et localement libre sur $W$ de rang $\geq 1$). Alors $\cM$ est $Z$-clos pour $Z:=X-W$.
\end{Lem}

\begin{proof}
D'après \cite{EGA IV}${}_2$ 5.10.5, il s'agit de voir que 
$$\prof_{Z}(\cM)\geq 2.$$
Soit $z\in Z$. Le point $f(z)$ n'est pas dans l'ouvert schématiquement dense $U$ de $S$, ce n'est donc pas un point associé à $S$, autrement dit $\prof(\cO_{S,f(z)})\geq 1$. De plus, le point $z$ n'est pas un point générique de $X_s$, et n'est donc pas associé à $\cM\otimes k(s)$ d'après les hypothèses, d'où $\prof((\cM\otimes k(s))_z)\geq 1$. Conclusion, $\prof(\cM_z)\geq 2$ (\cite{EGA IV}${}_2$ 6.3.1).
\end{proof}
Le lemme \ref{profondeur} permet également de ramener la preuve du théorème \ref{composantes} au cas où $X$ est un schéma affine. 

En effet, supposons d'abord qu'il existe un recouvrement ouvert $X=X_1\cup X_2$ tel que le théorème soit vrai pour les morphismes induits $f_i:X_i\ra S$ par $f$ et les modules $\cM|_{X_i}$, pour $i=1,2$. Il existe alors des éclatements $U$-admissibles $\varphi_i:S_{i}'\ra S$ d'idéaux cohérents $\cI_i$ de $\cO_S$ et des modules $\widetilde{\cM}_i$ qui sont solutions du problème pour les $f_i$. Soit $\varphi:S''\ra S$ l'éclatement $U$-admissible du produit $\cI_1\cdot\cI_2$ dans $S$. D'après \cite{GR} 5.1.2 (v), on a $\varphi=\gamma_i\circ\varphi_i$ où $\gamma_i:S''\ra S_{i}'$ est l'éclatement $(\varphi_i)^{-1}(U)$-admissible de $(\prod_{j\neq i}\cI_j)\cO_{S_{i}'}$ dans $S_{i}'$. D'où des carrés cartésiens
\begin{displaymath}
\xymatrix{
X_{i} \ar[d]^{f_i} & X_{i}' \ar[l] \ar[d]^{f_{i}'} & X_{i}'' \ar[l]_{\Gamma_i} \ar[d]^{f_{i}''}\\
S        & S_{i}' \ar[l]_{\varphi_i}        & S''. \ar[l]_{\gamma_i} \ar@/^1pc/[ll]^{\varphi}
}
\end{displaymath}
Le pull-back $f'':X''\ra S''$ de $f:X\ra S$ par $\varphi$ est alors le recollement de $f_{1}''$ et $f_{2}''$. On pose
$$\widetilde{\cM}|_{X_{i}''}:=(\Gamma_{i})^{*}\widetilde{\cM}_i.$$
Ce $\cO_{X_{i}''}$-module coïncide avec le $\cO_{X''}$-module $\cM''$ sur 
$$(X_{U''}''\cup V'')\cap X_{i}''.$$
En particulier $\widetilde{\cM}|_{X_{i}''}$ est localement libre sur cet ouvert, de rang $\geq 1$. Mais comme il est sans composante immergée sur les fibres de $f_{i}''$, il résulte du lemme \ref{profondeur} qu'il est $Z''\cap X_{i}''$-clos pour $Z''=X''-X_{U''}''\cup V''$. Par conséquent, les faisceaux $\widetilde{\cM}|_{X_{i}''}$ pour $i=1,2$ se recollent en un $\cO_{X''}$-module $\widetilde{\cM}$, qui est solution du problème pour le module $\cM$.

Par récurrence, on en déduit que si $X=\cup_{i=1}^n X_i$ pour un certain $n\geq 1$ de sorte que le théorème soit vrai pour les $X_i/S$ et les $\cM|_{X_i}$, alors le théorème est vrai pour $X/S$ et $\cM$. Comme le schéma $X$ est quasi-compact, on est ainsi ramené à la

\begin{proof}[Démonstration du théorème \ref{composantes} lorsque $X$ est affine.]
Soit $\cM^*$ le $\cO_{X}$-module dual de $\cM$. C'est un module cohérent sur le schéma affine $X$. Il existe donc un entier $r\geq 1$ et un morphisme $\cO_X$-linéaire surjectif $u:\cO_{X}^r\twoheadrightarrow \cM^*$. En composant le morphisme injectif dual $\cM^{**}\hookrightarrow\cO_{X}^r$ avec le morphisme canonique $c:\cM\ra\cM^{**}$, on obtient un morphisme $v:\cM\ra\cO_{X}^r$.

Par hypothèse, le module $\cM$  est localement libre au-dessus de $U$. Il en va donc de même pour le noyau de $u$, puis pour le conoyau $\cN$ de $v$. En particulier, le module $\cN|_{X_U}$ est $U$-plat.

On peut donc appliquer le théorème 5.2.2 de \cite{GR}, pour obtenir un éclatement $U$-admissible $S'\ra S$ tel que le transformé strict $\overline{\cN'}$ de $\cN$ sur $X'$ soit \emph{$S'$-plat}. Le $\cO_{X'}$-module $\overline{\cN'}$ est quotient du pull-back $\cN'$ de $\cN$. En particulier, il est cohérent et le morphisme surjectif $\cO_{X'}^{r}\twoheadrightarrow\cN'$ induit un morphisme surjectif $\cO_{X'}^{r}\twoheadrightarrow\overline{\cN'}$. Soit $\widetilde{\cM}$ son noyau, cohérent et $S'$-plat. Pour tout $s'\in S'$, on a une suite exacte
$$0\ra \widetilde{\cM}\otimes k(s') \ra \cO_{X'_{s'}}^r \ra \overline{\cN'}\otimes k(s')\ra 0.$$ Le $S'$-schéma $X'\ra S'$ étant à fibres réduites, on en déduit que $\widetilde{\cM}$ est sans composante immergée sur les fibres de $X'/S'$.

Montrons enfin que $\cM'=\widetilde{\cM}$ sur l'ouvert $X_{U'}'\cup V'$. Comme par hypothèse $\cM$ est localement libre non seulement sur $X_U$, mais aussi sur la réunion $X_U\cup V$, le conoyau $\cN$ de $v$ est localement libre sur celle-ci. Le morphisme $v':\cM'\ra\cO_{X'}^r$ est donc injectif sur $X_{U'}'\cup V'$. De plus, par définition du transformé strict, le pull-back $\cN'$ de $\cN$ coïncide avec $\overline{\cN'}$ sur cet ouvert (noter que l'ouvert $U'$ de $S'$ est schématiquement dense, comme rappelé plus haut). D'où $\cM'=\widetilde{\cM}$ sur  $X_{U'}'\cup V'$ par définition de $\widetilde{\cM}$. 
\end{proof}

\begin{proof}[Démonstration du théorème \ref{prolongement}.]
Grâce au théorème de platification \cite{GR} 5.2.2, on peut supposer $\cM$ plat sur $S$, et donc inversible sur un ouvert $S$-dense de $X/S$ (\emph{loc. cit} 2.1). Il suffit alors d'appliquer \ref{composantes} puis \ref{clé}. 
\end{proof}

\begin{Rem}
\emph{Soit $\pi:X\ra S$ un morphisme régulier de schémas noethériens, muni d'une section $\sigma$. Si $S$ est excellent et réduit, et si $X$ est le seul voisinage de $\sigma(S)$ dans $X$, alors Boutot démontre un énoncé analogue à celui du théorème \ref{prolongement}, utile pour son étude du schéma de Picard local (\cite{Bout} V 2.4).}
\end{Rem}

\subsection{La situation universelle : énoncé du théorème et esquisse de la preuve} \label{suep}

Si $X$ est un schéma plat sur un trait $S$ de corps de fonctions $K$, les points associés à $X$ sont les mêmes que les points associés à la fibre générique $X_K$ (\cite{EGA IV}${}_2$ 3.3.1). En particulier, si le schéma $X$ est localement noethérien, il est réduit (resp. intègre) si et seulement si $X_K$ l'est (\emph{loc. cit.} 3.2.1). 

Par ailleurs, si $S$ est un trait de corps de fonctions $K$, tout schéma propre sur $K$ possède un modèle propre et plat sur $S$, d'après la version relative du théorème de compactification de Nagata (cf. Deligne \cite{D}, Conrad \cite{C} et Lütkebohmert \cite{Lü}). L'existence de modèles normaux semi-factoriels sur $S$ est donc conséquence de l'énoncé suivant.

\begin{Th} \label{T}
Soient $S^{\sh}\ra T\ra S$ des extensions fidèlement plates de traits, telles que la composée $S^{\sh}\ra S$ soit une hensélisation stricte de $S$. Soit $X_K$ un $K$-schéma propre géométriquement normal. 

Pour tout modèle propre et plat $X/S$ de $X_K$, il existe un éclatement $X'\ra X$ centré dans la fibre fermée de $X/S$, tel que $(X')_T$ soit semi-factoriel sur $T$. Le normalisé $\widetilde{X}$ de $X'$ est un modèle propre et plat de $X_K$ sur $S$, \emph{tel que $(\widetilde{X})_T$ soit normal et semi-factoriel sur $T$.}
\end{Th}

\begin{Rem} \label{Remreg}
\emph{On peut demander en outre que le centre de l'éclatement $X'\ra X$ ne rencontre pas le lieu régulier de $X$ (cf. la preuve section \ref{sup}). Ce raffinement nous sera utile section \ref{scsf}.}
\end{Rem}

\begin{Rem} \label{cdn}
\emph{Il se peut qu'il n'y ait << pas beaucoup >> de faisceaux inversibles sur $X_K$. Par exemple, si $K$ est un corps de nombres, le théorème de Mordell-Weil pour les variétés abéliennes (et la finitude du groupe de Néron-Séveri de $X_K$) montrent que le groupe $\Pic(X_K)$ est de type fini. Par conséquent, si on prend pour $S$ le spectre du localisé de l'anneau des entiers de $K$ en un idéal maximal, alors pour que la flèche $\Pic(X)\ra\Pic(X_K)$ soit surjective, il suffit que son image contienne un certain ensemble \emph{fini}. Au contraire, considérons un trait strictement hensélien $S$. Supposons de plus $X_K$ géométriquement connexe, de sorte que $\Pic(X_K)=\Pic_{X_K/K}(K)$ (cf. lemme \ref{extension}). Alors, dès que la variété de Picard $A_K$ de $X_K$ est non triviale, le groupe $\Pic(X_{K})$ n'est \emph{pas} de type fini. En effet, la fibre spéciale du modèle de Néron $A$ de $A_K$ sur $S$ est un groupe algébrique lisse de dimension non nulle sur le corps résiduel $k$ de $\cO(S)$, qui est séparablement clos. En particulier, le groupe $A(k)$ n'est pas de type fini, bien qu'il soit quotient de $A_{K}(K)$. Il est donc intéressant de disposer de modèles semi-factoriels qui \emph{commutent} à une extension de traits telle que l'hensélisation stricte.}
\end{Rem}

Le principe de la démonstration du théorème \ref{T}, disons dans le cas $S=T$, est le suivant. Supposons pour fixer les idées que l'on dispose d'un $K$-schéma lisse de type fini $\Lambda_K$ et d'un faisceau inversible $\cM_K$ sur $X_K\times_K\Lambda_K$, tel que pour tout faisceau inversible $\cL_{K}$ sur $X_{K}$, il existe $\lambda_{K}\in\Lambda_K(K)$ vérifiant $\cL_{K}\simeq(1\times_K\lambda_{K})^*\cM_K$. Supposons en outre que $\Lambda_K$ admette un modèle $\Lambda/S$ tel que
\begin{enumerate}[a)]
\item $\Lambda\ra S$ est lisse de type fini ;
\item tout $K$-point de $\Lambda_K$ se prolonge en une $S$-section de $\Lambda$.
\end{enumerate}
Prolongeons alors $\cM_K$ en un module cohérent $\cM$ sur $X\times_S\Lambda$. Grâce au point a), on peut appliquer le théorème \ref{prolongement} au morphisme $p_1:X\times_S\Lambda\ra X$, avec l'ouvert schématiquement dense $X_K$ de $X$ et le module $\cM$. On obtient alors un carré cartésien
\begin{displaymath}
\xymatrix{
X\times_S \Lambda\ar[d]^{p_1} &\ar[l] X'\times_S \Lambda \ar[d]^{p_1} \\
X & \ar[l] X'
}
\end{displaymath}
où $X'\ra X$ est un éclatement centré dans la fibre fermée de $X/S$, et un module \emph{inversible} $\widetilde{\cM}$ sur $X'\times_S\Lambda$ qui prolonge $\cM_K$. Si maintenant $\lambda_{K}$ paramètre un faisceau inversible $\cL_{K}$ sur $X_{K}$, tout $\lambda\in \Lambda(S)$ prolongeant $\lambda_{K}$ (point b)) fournit un prolongement inversible de $\cL_{K}$ sur $X'$, à savoir $(1\times_S\lambda)^*\widetilde{\cM}$. Le $S$-modèle $X'$ de $X_K$ est donc semi-factoriel. \emph{A fortiori}, son normalisé $(X')^{\nor}$ est semi-factoriel sur $S$, et on verra (lemme \ref{norm}) qu'il est propre sur $S$.

Maintenant, les faisceaux inversibles sur $X_{K}$ ne forment pas une famille \emph{limitée}, au sens où ils ne sont pas paramétrés par un $K$-schéma \emph{de type fini} : le $K$-schéma $\Pic_{X_K/K}$ est seulement \emph{localement} de type fini. Cependant, comme mentionné dans la remarque ci-dessus, on peut utiliser le théorème de finitude du groupe de Néron-Séveri géométrique de $X_K$. 

Une autre difficulté provient du fait qu'en l'absence de point rationnel sur $X_K$, \emph{il n'y a pas en général de faisceau inversible universel sur $X_K\times_K\Pic_{X_K/K}$}. Pour contourner ce problème, il faut remplacer le foncteur de Picard de $X_K/K$ par un foncteur de Picard \emph{rigidifié} (cf. \cite{R} \no 2). On récupère alors un faisceau inversible universel (rigidifié) avec lequel travailler. Cependant, en contrepartie, on perd en finitude : la variété de Picard rigidifiée de $X_K$ n'est plus une variété abélienne, et par conséquent ne possède pas en général de modèle de Néron (de type fini) sur $S$, qui aurait été un bon atout pour construire $\Lambda$. Néanmoins, en choisissant bien le rigidificateur de $X_K$, on peut assurer que la variété de Picard rigidifiée correspondante soit \emph{semi-abélienne}, et en particulier possède un modèle de Néron au sens suivant.

\begin{Def} \label{defmn}
Soient $S$ un trait de corps de fonctions $K$ et $N_K$ un $K$-schéma lisse et séparé. Un modéle de Néron de $N_K$ sur $S$ est un $S$-schéma lisse et séparé de fibre générique $N_K$, vérifiant la propriété de Néron d'extension des morphismes :

pour tout $S$-schéma \emph{lisse} $Y$, la flèche de restriction
$$\Hom_S(Y,N)\ra\Hom_K(Y_K,N_K)$$
est bijective. 
\end{Def}
Le schéma $N$ ainsi défini est un modèle de Néron $\ltf$ de $N_K$ sur $S$ au sens de \cite{BLR} 10.1/1 : il est localement de type fini sur $S$ (car lisse sur $S$), mais il n'est pas de type fini sur $S$ en général. Voir \emph{loc. cit.} 10.2/2 pour des conditions nécessaires et suffisantes d'existence d'un tel $N$ lorsque $N_K$ est un $K$-schéma \emph{en groupes}.

En utilisant que le groupe abstrait des composantes géométriquement connexes de la fibre spéciale du modèle de Néron d'une variété semi-abélienne est \emph{de type fini}, on parvient finalement à construire un espace de paramètres $\Lambda/S$ convenable.

\subsection{La situation universelle : la preuve} \label{sup}

Commençons par un lemme général.

\begin{Lem}
Soit $X_K$ un schéma réduit propre sur un corps $K$. Soit $x_K$ un sous-schéma fermé de $X_K$, qui rencontre chaque composante connexe de $X_K$. Alors $x_K$ est un \emph{rigidificateur de $P_K$} au sens \cite{R} 2.1.1.
\end{Lem}
  
\begin{proof}
En effet, notons $\{X_{K,c}|\in C\}$ l'ensemble des composantes connexes de $X_K$, et $x_{K,c}$ la restriction de $x_K$ à $X_{K,c}$. Considérons le pull-back
$$\Gamma(X_K)=\prod_{c\in C}\Gamma(X_{K,c})\ra\Gamma(x_K)=\prod_{c\in C}\Gamma(x_{K,c}).$$
Comme $X_K$ est réduit et les $x_{K,c}$ non vides, les morphismes d'anneaux 
$$\Gamma(X_{K,c})\ra\Gamma(x_{K,c})$$ 
ont pour source un corps et pour but un anneau non nul, donc sont injectifs. Le lemme résulte alors de \cite{R} 2.2.2.
\end{proof}  

A partir de maintenant et jusqu'à la fin de cette sous-section \ref{sup}, on se place sous les hypothèses du théorème \ref{T}. On note $L/K$ l'extension générique de $T/K$, et on pose $P_K:=\Pic_{X_K/K}$.

\paragraph{1) Paramétrage sur $K$ des faisceaux inversibles algébriquement équivalents à zéro sur $X_L$.} Fixons un sous-schéma fermé $x_K\hookrightarrow X_K$, dont la restriction à chaque composante connexe de $X_K$ est un point fermé. Le foncteur de Picard rigidifié $(P_K,x_K)$ est représentable par un $K$-schéma en groupes localement de type fini (\cite{R} 2.3.1). Le morphisme d'oubli 
$$(P_K,x_K) \ra P_K$$
est couvrant pour la toplogie étale (\emph{loc. cit.} 2.1.2 b)). De plus, comme le lieu lisse de $X_K$ est ici un ouvert dense de $X_K$, on peut choisir $x_K$ \emph{étale sur $K$}. Dans ce cas, le noyau du morphisme d'oubli est un \emph{tore} $M_K$ (\emph{loc. cit.} 2.4.3 b)). On a ainsi une suite exacte (de faisceaux abéliens pour la topologie étale) 
$$0\ra M_K \ra (P_K ,x_K)^0 \ra P_{K}^0 \ra 0,$$
où ${}^0$ désigne les composantes neutres. L'hypothèse que $X_K$ est géométriquement normal assure que le $K$-schéma $P_{K}^0$ est propre (\cite{G} 2.1 (ii)). En particulier, les $K$-schémas en groupes $P_{K}^0$ et $(P_K,x_K)^0$ sont << sans composante additive >>. Il résulte donc de \emph{loc. cit.} 3.1 que leurs réduits $P_{K,\red}^0$ et $(P_K ,x_K)_{\red}^0$ sont des $K$-schémas \emph{en groupes} (commutatifs), lisses et connexes. On en déduit en particulier un homomorphisme 
$$(P_K ,x_K)_{\red}^0 \ra P_{K,\red}^0$$ 
de noyau $M_K$, et $(P_K ,x_K)_{\red}^0$ est ainsi extension d'une variété abélienne par un tore. D'après \cite{BLR} 10.2/2, le $K$-schéma en groupes $(P_K ,x_K)_{\red}^0$ possède donc un modèle de Néron $N_0$ sur $S$ (définition \ref{defmn}). 

Soit $\Phi_{N_0}$ le $k$-schéma en groupes étale des composantes connexes de la fibre spéciale de $N_0$. Notant $l$ le corps résiduel au point fermé de $T$, l'image de la spécialisation
$$N_0(T)\ra \Phi_{N_0}(l)$$
est un sous-groupe de $\Phi_{N_0}(l)$. Il est donc de type fini (cf. par exemple \cite{BX} 4.11 (i)), et on peut en choisir un système fini de générateurs $\Gamma$. On en fixe un relèvement $\{\lambda_{L,0}^{\gamma}\in(P_K ,x_K)_{\red}^0(L)\ |\ \gamma\in\Gamma\}$. Puis pour tout $\gamma\in\Gamma$, on note $N_{0}^\gamma$ la réunion de la fibre générique de $N_0$ et de la composante connexe de la fibre spéciale de $N_0$ dans laquelle se spécialise $\lambda_{L,0}^\gamma$. On note aussi $N_{0}^0$ la composante neutre de $N_0$. La réunion 
$$\Lambda_0:=N_{0}^0\bigcup_{\gamma\in\Gamma}N_{0}^\gamma$$
est alors un ouvert de $N_0$, de fibre générique $(P_K ,x_K)_{\red}^0$.

Maintenant, l'objet universel pour $(P_K,x_K)$ est représenté par un couple $(\cP_K,t)$, formé d'un faisceau inversible $\cP_K$ sur $X_K\times_K (P_K,x_K)$ et d'une trivialisation $t$ de $\cP_K$ le long du rigidificateur $x_K\times_K (P_K,x_K)$. D'autre part, comme le schéma $X$ est de type fini sur un trait, son lieu de régularité $X^{\reg}$ est un ouvert de $X$. Et $\Lambda_0$ étant lisse sur $S$, l'ouvert $X^{\reg}\times_S\Lambda_0$ de $X\times_S\Lambda_0$ est régulier. On peut donc prolonger le faisceau inversible  
$$\cP_K|_{X_K\times_K(P_K ,x_K)_{\red}^0}$$ 
en un faisceau inversible sur $X_K\times_K\Lambda_K\cup X^{\reg}\times_S\Lambda_0$. Puis on peut prolonger le faisceau obtenu en un module \emph{cohérent} $\cM_0$ sur $X\times_S\Lambda_0$ (\cite{EGA I} 9.4.8).

\paragraph{2) Paramétrage sur $K$ du groupe de Néron-Severi de $X_L$.}
Soit $\overline{K}$ une clôture algébrique de $K$. Le groupe de Néron-Severi géométrique de $X_K$ $$\NS_{g}:=P_K(\overline{K})/P_{K}^0(\overline{K})$$
est de type fini (\cite{SGA6} XIII 5.1). Comme $M_K$ est géométriquement connexe, ce groupe coïncide avec le << groupe de Néron-Severi géométrique rigidifié >> de $X_K$ :
$$\NS_{g,r}:=(P_K,x_K)(\overline{K})/(P_K,x_K)^0(\overline{K}).$$
En particulier, le << groupe de Néron-Severi rigidifié >> de $X_L$
$$\NS_{r,L}:=(P_K,x_K)(L)/(P_K,x_K)^0(L)$$
est de type fini. Il existe donc un ensemble fini d'indices $I$ et des points $\{\lambda_{L,i}\in (P_K,x_K)(L)\ |\ i\in I\}$ tels que $\NS_{r,L}$ soit engendré par les classes des $\lambda_{L,i}$. 

Par hypothèse, l'extension $T/S$ est une sous-extension d'une hensélisation stricte $S^{\sh}/S$. Le trait $T$ est donc limite projective filtrante de traits $S_{\alpha}$ étales sur $S$. Notant $K_{\alpha}$ le corps de fonctions de $S_{\alpha}$, le corps $L$ est alors limite inductive filtrante des $K_{\alpha}$. Il existe donc $\alpha$ tel que pour tout $i\in I$, l'élément $\lambda_{L,i}\in (P_K,x_K)(L)$ provienne d'un élément $\lambda_{K_{\alpha},i}\in (P_K,x_K)(K_{\alpha})$. Pour $i\in I$, posons
$$\cL_{K_\alpha,i}:=(1\times_K\lambda_{K_{\alpha},i})^*\cP_K,$$
faisceau inversible sur $X_K\times_K K_\alpha$. Prolongeons le en un faisceau inversible sur $X_K\times_K K_{\alpha}\cup X^{\reg}\times_S S_{\alpha}$. Puis choisissons un prolongement cohérent $\cM_{\alpha,i}$ de celui-ci  sur $X\times_S S_{\alpha}$. 

\paragraph{3) Modification de $X$ en dehors de $X_K$.}
Considérons l'union disjointe de $\Lambda_0$ et de copies $S_{\alpha,i}$ de $S_{\alpha}$ indexées par $I$ :
$$\Lambda:=\Lambda_0\coprod_{i\in I}S_{\alpha,i}.$$
C'est un $S$-schéma lisse de type fini. La première projection $p_1:X\times_S \Lambda\ra X$ est donc un morphisme lisse de type fini. De plus, on a construit un module cohérent $\cM$ sur $X\times_S \Lambda$, égal à $\cM_0$ sur $X\times_S \Lambda_0$ et à $\cM_{\alpha,i}$ sur $X\times_S S_{\alpha,i}$. Il est inversible au-dessus de $X_K\cup X^{\reg}$. 

Appliquons le théorème \ref{prolongement} au morphisme $p_1$ et au module $\cM$, avec l'ouvert schématiquement dense $X_K\cup X^{\reg}$ de $X$. On trouve un carré cartésien
\begin{displaymath}
\xymatrix{
X\times_S \Lambda\ar[d]^{p_1} &\ar[l] X'\times_S \Lambda \ar[d]^{p_1} \\
X & \ar[l] X'
}
\end{displaymath}
où $X'\ra X$ est un éclatement $(X_K\cup X^{\reg})$-admissible, et un $\cO_{X'\times_S \Lambda}$-module \emph{inversible} $\widetilde{\cM}$ qui prolonge $\cM_K$. Le modèle $X'/S$ de $X_K$ ainsi obtenu est encore propre, et plat puisque $X_K$ est schématiquement dense dans $X'$. 

\paragraph{4) Le schéma $(X')_T$ est semi-factoriel sur $T$.}
 
Soit $\cL_L$ un faisceau inversible sur $X_L$. Il s'agit de voir qu'il existe un faisceau inversible  $\cL_T$ sur $X'\times_S T$ tel que $\cL_T\otimes L$ soit isomorphe à $\cL_L$. 

On peut toujours rigidifier $\cL_L$ le long du schéma artinien $x_K\otimes_K L$ pour obtenir un point $\lambda_L\in (P_K,x_K)(L)$. D'après l'étape \textbf{2)}, il existe $\{n_i\in\bbZ\ |\ i\in I\}$ et $\lambda_{L,0}\in (P_K,x_K)^0(L)$ tels que
$$\lambda_{L}=\sum_{i\in I}n_i\cdot\lambda_{L,i}+\lambda_{L,0}.$$
Or $(P_K,x_K)^0(L)=(P_K,x_K)_{\red}^0(L)$. D'après l'étape \textbf{1)}, il existe donc $\{m_{\gamma}\in\bbZ\ |\ \gamma\in\Gamma\}$ et $\lambda_{L,0}^0\in (P_K,x_K)_{\red}^0(L)$ se prolongeant dans $N_{0}^0(T)$ tels que
$$\lambda_{L,0}=\sum_{\gamma\in \Gamma_l}m_{\gamma}\cdot\lambda_{L,0}^{\gamma}+\lambda_{L,0}^0.$$
Utilisons maintenant l'étape \textbf{3)}. Par définition de $\Lambda_0\subseteq\Lambda$, il existe pour tout $\gamma\in\Gamma$ un élément $\lambda_{0}^{\gamma}\in\Lambda(T)$ prolongeant $\lambda_{L,0}^{\gamma}$ et un élément $\lambda_{0}^{0}\in\Lambda(T)$ prolongeant $\lambda_{L,0}^{0}$. Le faisceau inversible
$$\cL_{T,0}:=\otimes_{\gamma\in \Gamma}\big((1\times_S \lambda_{0}^\gamma)^*\widetilde{\cM}\big)^{\otimes m_{\gamma}}\otimes(1\times_S\lambda_{0}^0)^*\widetilde{\cM}$$
prolonge $(1\times_K\lambda_{L,0})^*\cP_K$ sur $X'\times_S T$.
Par définition de $\Lambda$, il existe pour tout $i\in I$ un élément $\lambda_i\in\Lambda(T)$ tel que 
$(1\times_S\lambda_i)^*\widetilde{\cM}$ prolonge $(1\times_K\lambda_{L,i})^*\cP_K$ sur $X'\times_S T$. Le faisceau inversible
$$\cL_{T}:=\otimes_{i\in I}\big((1\times_S\lambda_i)^*\widetilde{\cM}\big)^{\otimes n_i}\otimes\cL_{T,0}$$
prolonge $(1\times_K\lambda_{L})^*\cP_K\simeq\cL_L$ sur $X'\times_S T$.

\paragraph{5) Normalisation de $X'$.}
Pour achever la preuve du théorème \ref{T}, il ne reste plus qu'à montrer que le morphisme de normalisation  $\widetilde{X}:=(X')^{\nor}\ra X'$ est fini.

\begin{Lem}\label{norm}
Soit $X$ un schéma de type fini et plat sur un trait $S$. Supposons la fibre générique de $X/S$ géométriquement normale (en particulier $X$ est réduit). Alors le morphisme de normalisation $X^{\nor}\ra X$ est fini.
\end{Lem}

\begin{proof}
Soit $\widehat{S}$ le spectre du complété de l'anneau local $\cO(S)$ pour la topologie $\fm$-adique.
Comme c'est un trait excellent, le morphisme de normalisation $(X_{\widehat{S}})^{\nor}\ra X_{\widehat{S}}$ est fini. Par descente $\fpqc$ pour les morphismes finis, il suffit donc de montrer que le carré commutatif 
\begin{displaymath}
\xymatrix{
(X_{\widehat{S}})^{\nor}\ar[r]\ar[d] & X_{\widehat{S}} \ar[d]\\
X^{\nor}\ar[r] & X
}
\end{displaymath}
est cartésien. Pour cela, on peut supposer $X$ intègre et affine.

Notons $X_K$ (resp. $X_{\widehat{K}}$) la fibre générique de $X/S$ (resp. $X_{\widehat{S}}/\widehat{S}$). Ces fibres sont normales par hypothèse. La fermeture intégrale $X^{\nor}$ de $X$ (resp. $(X_{\widehat{S}})^{\nor}$ de $X_{\widehat{S}}$) relativement à son algèbre des fonctions rationnelles coïncide donc avec sa fermeture intégrale relativement à $\cO_{X_K}$ (resp. $\cO_{X_{\widehat{K}}}$). De plus, notant $k$ le corps résiduel de $S$ au pont fermé, la projection $X_{\widehat{S}}\ra X$ induit un isomorphisme 
$$X_k\times_S\widehat{S}\ra X_k.$$
Comme de plus $X_K$ est quasi-compact, on peut appliquer le résultat de descente de \cite{FR} 4.3, pour obtenir que le diagramme ci-dessus est cartésien. 
\end{proof}

\subsection{Changement de trait} \label{changement}

Du théorème \ref{T}, et du fait que le composé de deux éclatements admissibles (d'un schéma quasi-compact et quasi-séparé) soit encore un éclatement admissible (\cite{GR} 5.1.4), on déduit le 

\begin{Cor} \label{Ts}
Soit $S^{\sh}\ra S$ une hensélisation stricte d'un trait $S$. Soient $n$ un entier $\geq 1$, et $T_i/S$, $i=1,\ldots,n$, des sous-extensions de traits fidèlement plates de $S^{\sh}/S$. Soit $X_K$ un $K$-schéma propre géométriquement normal.

Pour tout modèle propre et plat $X/S$ de $X_K$, il existe un morphisme $\widetilde{X}\ra X$, composé d'un éclatement centré dans la fibre fermée de $X/S$ suivi de sa normalisation, tel que $\widetilde{X}/S$ soit un modèle propre et plat de $X_K$ sur $S$, \emph{normal et semi-factoriel après chacun des changements de trait $T_i\ra S$.}
\end{Cor}

On obtient ainsi des modèles semi-factoriels sur $S$ qui commutent à certaines extensions de traits \emph{prescrites} et \emph{en nombre fini}. Nous allons voir que l'on peut construire des modèles semi-factoriels sur $S$ qui commutent à une certaine \emph{classe} d'extensions de traits. 

\begin{Def} \label{permis}
Soit $T\ra S$ un morphisme de traits. On dira que $T\ra S$ est \emph{permis} si les conditions suivantes sont vérifiées : le trait $T$ est strictement hensélien, le morphisme $T\ra S$ est régulier, et le pull-back $\cO(S)\ra\cO(T)$ envoie uniformisante sur uniformisante.
\end{Def}

Notant $K$, $L$ les corps de fonctions $S$, $T$, et $k$, $l$ les corps résiduels aux points fermés de $S$, $T$, le morphisme $T\ra S$ est régulier s'il est plat, et si les deux extensions $L/K$ et $l/k$ sont séparables (non nécessairement algébriques) (\cite{EGA IV}${}_2$ 6.8.1 et 6.7.6). Par exemple, si $S$ est excellent et $k$ séparablement clos, le spectre $\widehat{S}$ du complété de $\cO(S)$ pour la topologie $\fm$-adique définit une extension permise $\widehat{S}\ra S$.

\begin{Rem}
\emph{Si $T\ra S$ est une extension permise de traits, elle est en particulier \emph{d'indice de ramification $1$} au sens de \cite{BLR} 3.6/1. On demande en plus ici que l'extension des corps de fonctions soit séparable, et que $T$ soit strictement hensélien.}
\end{Rem}

\begin{Rem}\label{normal}
\emph{Si $T\ra S$ est une extension régulière de traits, telle que le pull-back $\cO(S)\ra \cO(T)$ envoit uniformisante sur uniformisante, alors pour tout schéma \emph{normal} (resp. \emph{régulier}) $X$ localement de type fini sur $S$, le schéma $X_T:=X\times_S T$ est encore \emph{normal} (resp. \emph{régulier}). Cela résulte de \cite{EGA IV}${}_2$ 6.5.4 (ii) et 6.7.4.1 (resp. 6.5.2 (ii) et 6.7.4.1).}
\end{Rem}

\begin{Th} \label{universel}
Soient $S$ un trait de corps de fonctions $K$ et $X_K$ un $K$-schéma propre géométriquement normal. 

Pour tout modèle propre et plat $X/S$ de $X_K$, il existe un morphisme $\widetilde{X}\ra X$, composé d'un éclatement centré dans la fibre fermée de $X/S$ suivi de sa normalisation, tel que $\widetilde{X}/S$ soit un modèle propre et plat de $X_K$, \emph{normal et semi-factoriel sur $S$ et après toute extension permise de traits $T\ra S$}.
\end{Th}
Soit $X/S$ un modèle propre et plat de $X_K$. La construction de la sous-section \ref{sup} s'applique en particulier lorsque $T$ est un hensélisé strict de $S$, fournissant une flèche $X'\ra X$. Le modèle $X'$ ainsi construit est semi-factoriel après toute extension permise de $S$. Pour le voir, nous avons besoin du lemme d'approximation suivant.

\begin{Lem} \label{desc}
Soient $S$ un trait de corps de fonctions $K$ et $S^{\sh}$ un hensélisé strict de $S$ de corps de fonctions $K^{\sh}$. Soit $G_K$ un $K$-schéma en groupes localement de type fini, dont la composante neutre réduite $G_{K,\red}^0$ est lisse sur $K$. Soit $Z_{K^{\sh}}$ une composante connexe de $G_{K^{\sh}}$. 

Considérons une extension permise $T/S$ et notons $L$ le corps de fonctions de $T$. Comme $T$ est strictement hensélien, on peut fixer une factorisation $T\ra S^{\sh}\ra S$, induisant des plongements $K\subseteq K^{\sh}\subseteq L$. Alors, si l'ensemble $Z_{K^{\sh}}(L)$ est non vide, l'ensemble $Z_{K^{\sh}}(K^{\sh})$ l'est aussi.
\end{Lem}

\begin{proof}
Les deux extensions résiduelles de $S^{\sh}/S$ sont algébriques. D'après \cite{BA} Chap. 5, \S 15, n${}^\circ$4, Corollaire 3, l'extension $T/S^{\sh}$ est donc permise. Ceci permet de se ramener au cas où $S=S^{\sh}$. 

Supposons $Z_K(L)$ non vide. Le réduit $Z_{K,\red}$ est alors lisse sur $K$. En effet, comme l'extension $L/K$ est séparable, on a $(Z_{K,\red})_L=Z_{L,\red}$ (\cite{EGA IV}${}_2$ 4.6.4). Or, comme $Z_K(L)$ est non vide, il existe une composante connexe $Z_{L}^0$ de $Z_L$ qui est un torseur trivial sous $G_{L}^0$. En particulier, le réduit $Z_{L,\red}^0$ est $L$-isomorphe au réduit $G_{L,\red}^0=(G_{K,\red}^0)_L$, donc est lisse sur $L$. Comme $Z_{L,\red}^0$ est une composante connexe de $(Z_{K,\red})_{L}$, la projection de ce dernier sur $Z_{K,\red}$ induit un morphisme
$$Z_{L,\red}^0\ra Z_{K,\red}.$$
Ce morphisme est plat, et surjectif : la composante connexe $Z_{L,\red}^0$ est définie sur une sous-extension finie de $L/K$ (appliquer \cite{EGA IV}${}_2$ 4.9.7 en notant que $Z_{K,\red}$ est de type fini sur $K$), donc son image dans $Z_{K,\red}$ est à la fois ouverte et fermée. En conclusion, le schéma $Z_{K,\red}$ est bien lisse sur $K$ (\cite{EGA IV}${}_4$ 17.7.1 (ii)).

Ceci étant, choisissons $\lambda_L\in Z_K(L)=(Z_{K,\red})(L)$. En chassant les dénominateurs, on peut toujours trouver un $S$-schéma affine de type fini $U$ dont la fibre générique est un voisinage ouvert de l'image de $\lambda_L$ dans $Z_{K,\red}$, et tel que $\lambda_L$ se prolonge en un élément $\lambda\in U(T)$. En particulier, la fibre fermée de $U/S$ est non-vide. Maintenant, l'extension $T/S$ étant permise, elle est d'indice de ramification $1$ au sens de \cite{BLR} 3.6/1. Comme $U_K$ est lisse sur $K$, il existe donc un $S$-morphisme $U'\ra U$ composé d'un nombre fini de dilatations  centrées dans les fibres spéciales, tel que $\lambda$ se relève dans le lieu lisse de $U'/S$ (\emph{loc. cit.} 3.6/4). Mais alors $U'(S)$ est nécessairement non vide, et \emph{a fortiori} $Z_K(K)=(Z_{K,\red})(K)$ est non vide.
\end{proof}

\begin{proof}[Démonstration du théorème \ref{universel}.]
Conservons les constructions des étapes \textbf{1)}, \textbf{2)} et \textbf{3)} de la sous-section \ref{sup} dans le cas $T=S^{\sh}$. D'après l'étape \textbf{4)}, on sait déjà que le modèle $X'$ obtenu est semi-factoriel après l'extension $S^{\sh}/S$. Soit maintenant $T/S$ une extension permise de traits quelconque. Notons $L$ le corps de fonctions de $T$, et fixons une factorisation $T\ra S^{\sh}\ra S$, induisant des plongements $K\subseteq K^{\sh}\subseteq L$. Considérons un faisceau inversible $\cL_L$ sur $X_L$ et montrons qu'il se prolonge en un faisceau inversible sur $X'\times_S T$. 

Rigidifions $\cL_L$ le long de $x_K\otimes_K L$ pour obtenir un point $\lambda_L\in (P_K,x_K)(L)=(P_K,x_K)_{K^{\sh}}(L)$. Soit $Z_{K^{\sh}}$ la composante connexe de $(P_K,x_K)_{K^{\sh}}$ telle que $\lambda_L\in Z_{K^{\sh}}(L)$. D'après le lemme \ref{desc}, il existe $\lambda_{K^{\sh}}\in Z_{K^{\sh}}(K^{\sh})\subseteq (P_K,x_K)(K^{\sh})$. La différence $\delta_L:=\lambda_L-\lambda_{K^{\sh}}$ dans le groupe $(P_K,x_K)(L)$ appartient au sous-groupe $(P_K,x_K)^0(L)$. En effet, comme $Z_{K^{\sh}}(K^{\sh})$ est non vide, la composante connexe $Z_{K^{\sh}}$ de $(P_K,x_K)_{K^{\sh}}$ est \emph{géométriquement} connexe. Il en résulte que $\lambda_{K^{\sh}}$ et $\lambda_L$, vus comme $L$-points de $(P_K,x_K)_L$, sont dans la même composante connexe de $(P_K,x_K)_L$, à savoir $Z_L$. On a donc bien $\delta_L\in (P_K,x_K)_{L}^0(L)=(P_K,x_K)^0(L)$.

Maintenant, le faisceau inversible $(1\times_K\delta_L)^*\cP_K$ sur $X_L$ se prolonge en un faisceau inversible sur $X'\times_S T$. En effet, l'extension $T/S$ étant permise, elle est d'indice de ramification $1$ au sens de \cite{BLR} 3.6/1. D'après \emph{loc. cit.} 7.2/1, le modèle de Néron $N_0$ de $(P_K,x_K)_{\red}^0$ commute donc à l'extension $T/S$. Par conséquent, d'après l'étape \textbf{1)} de la sous-section \ref{sup} appliquée au cas de $S^{\sh}$, il existe $\{d_{\gamma}\in\bbZ\ |\ \gamma\in\Gamma\}$ et $\delta_{L}^0\in (P_K,x_K)_{\red}^0(L)$ se prolongeant en $\delta_{T}^0\in N_{0}^0(T)$
tels que
$$\delta_L=\sum_{\gamma\in \Gamma}d_{\gamma}\cdot\lambda_{K^{\sh},0}^{\gamma}+\delta_{L}^0.$$
Le faisceau inversible 
$$\cL_{T,\delta}:=\otimes_{\gamma\in \Gamma}\big((1\times_S \lambda_{0}^\gamma)^*\widetilde{\cM}\otimes_{S^{\sh}} T\big)^{\otimes d_{\gamma}}\otimes(1\times_S\delta_{T}^0)^*\widetilde{\cM}$$ 
prolonge alors $(1\times_K\delta_L)^*\cP_K$ sur $X'\times_S T$. 

Par ailleurs, comme $(X')_{S^{\sh}}$ est semi-factoriel sur $S^{\sh}$, il existe un faisceau inversible $\cL_{S^{\sh}}$ qui prolonge $(1\times_K\lambda_{K^{\sh}})^*\cP_K$ sur $(X')_{S^{\sh}}$. Finalement, le faisceau inversible
$$\cL_T:=(\cL_{S^{\sh}}\otimes_{S^{\sh}}T)\otimes \cL_{T,\delta}$$
prolonge $\big((1\times_K\lambda_{K^{\sh}})^*\cP_K\otimes_{K^{\sh}}L\big)\otimes(1\times_K\delta_L)^*\cP_K\simeq(1\times_K\lambda_L)^*\cP_K\simeq\cL_L$ sur $X'\times_S T$.

On a ainsi prouvé que le schéma $(X')_T$ est semi-factoriel sur $T$ pour toute extension permise $T/S$. Quitte à éclater $X'/S$ le long d'un sous-schéma fermé de sa fibre spéciale, on peut supposer que $X'$ est aussi semi-factoriel sur $S$ (théorème \ref{T}). Comme le morphisme de normalisation
$\widetilde{X}\ra X'$ est propre (lemme \ref{norm}), on obtient un morphisme $\widetilde{X}\ra X$ comme dans l'énoncé du théorème.
\end{proof}

\subsection{Compactifications semi-factorielles d'un schéma régulier} \label{scsf}

\begin{Def} \label{xcomp}
Soit $S$ un trait de corps de fonctions $K$. Soit $X$ un $S$-schéma de type fini et séparé. On appellera \emph{$X_K$-compactification de $X$} un $S$-schéma propre $\overline{X}$ muni d'une immersion ouverte $X\ra\overline{X}$ induisant un isomorphisme entre les fibres génériques.
\end{Def}

\begin{Th} \label{va}
Soient $S^{\sh}\ra T\ra S$ des extensions fidèlement plates de traits, telles que la composée $S^{\sh}\ra S$ soit une hensélisation stricte de $S$. Soit $X$ un $S$-schéma de type fini, séparé et plat, dont la fibre générique $X_K$ est un $K$-schéma propre géométriquement normal. On suppose le schéma $X$ \emph{régulier}.

Il existe une $X_K$-compactification $X\ra\overline{X}$ de $X$, avec $\overline{X}$ plat sur $S$, \emph{normal} et tel que \emph{la flèche de restriction $\Pic(\overline{X}_{T})\ra\Pic(X_{T})$ soit surjective.} (En particulier, le schéma $\overline{X}_T$ est normal et semi-factoriel sur $T$.)
\end{Th}

 \begin{proof}
D'après la version relative du théorème de compactification de Nagata, il existe un $S$-schéma propre $X_1$ et une $S$-immersion ouverte $X\ra X_1$. Quitte à remplacer $X_1$ par l'adhérence schématique de $X$ dans $X_1$, on peut supposer que $X\ra X_1$ est une $X_K$-compactification de $X$ (définition \ref{xcomp}) et que $X_1$ est plat sur $S$. En appliquant le théorème \ref{T} \emph{et} la remarque \ref{Remreg}, on trouve un éclatement $X$-admissible $(X_1)'\ra X_1$ tel que $(X_2)_T$ soit semi-factoriel sur $T$. Quitte à remplacer $X_1$ par $(X_1)'$, on peut donc supposer $(X_1)_T/T$ semi-factoriel.

Considérons maintenant les composantes irréductibles réduites $V_j$, $j=1,\ldots,\nu$, de la fibre spéciale de $X_{T}/T$. Le trait $T$ est limite projective filtrante de traits $S_\alpha$ étales sur $S$. Son corps résiduel $l$ au point fermé est donc limite inductive filtrante des corps résiduels $k_\alpha$ aux points fermés des $S_\alpha$. Les $V_{j}$ proviennent donc des composantes irréductibles réduites $V_{k_\alpha,j}$ de $X_{k_\alpha}$ pour un certain $\alpha$. Comme le schéma $X_{S_{\alpha}}$ est régulier, l'idéal cohérent $\cV_{k_\alpha,j}$ sur $X_{S_{\alpha}}$ définissant $V_{k_\alpha,j}\hookrightarrow X_{S_{\alpha}}$ est inversible. Son pull-back $\cV_{j}:=\cV_{k_\alpha,j}\otimes_{S_\alpha} T$ définit $V_{j}$ dans $X_T$. 

Prolongeons $\cV_{k_\alpha,1}$ en un module cohérent $\cM_1$ sur $X_1\times_S S_{\alpha}$. Puis appliquons le théorème \ref{prolongement} à la première projection $X_1\times_S S_{\alpha}\ra X_1$ et au module $\cM_1$, avec l'ouvert schématiquement dense $X$ de $X_1$. On obtient un carré cartésien
\begin{displaymath}
\xymatrix{
X_1\times_S S_{\alpha}\ar[d]^{p_1} &\ar[l] X_2\times_S S_{\alpha} \ar[d]^{p_1} \\
X_1 & \ar[l] X_2
}
\end{displaymath}
où $X_2\ra X_1$ est un éclatement $X$-admissible, et un module inversible $\widetilde{\cM_1}$ sur $X_2\times_S S_{\alpha}$ qui prolonge $\cV_{k_\alpha,1}$. L'immersion ouverte composée $X\ra X_2$ est une $X_K$-compactification de $X$, et $X_2$ est plat sur $S$. En procédant de même avec un prolongement cohérent de $\cV_{k_\alpha,2}$ sur $X_2\times_S S_{\alpha}$, etc, on obtient finalement une $X_K$-compactification $X\ra X_{\nu}=:X'$, avec $X'$ plat sur $S$, semi-factoriel sur $T$, et tel que tous les $\cV_{k_\alpha,j}$ se prolongent en des faisceaux inversibles sur $X'\times_S S_{\alpha}$. \emph{A fortiori}, les $\cV_j$ se prolongent en des faisceaux inversibles sur $(X')_T$.

L'homomorphisme de restriction $\Pic((X')_{T})\ra\Pic(X_{T})$  est surjectif. En effet, soit $\cL$ un faisceau inversible sur $X_{T}$. Comme $(X')_{T}/T$ est semi-factoriel, le faisceau $\cL_L:=\cL\otimes L$ se prolonge en un faisceau inversible $\cL'$ sur $(X')_{T}$. Sur l'ouvert régulier $X_T$, le faisceau inversible $\cL\otimes(\cL'|_{X_T})^{-1}$ est alors trivial en restriction à $X_L$, c'est-à-dire qu'il est isomorphe à un produit de puissances des $\cV_{j}$, $j=1,\ldots,\nu$. Mais comme ceux-ci se prolongent de façon inversible sur $(X')_T$, il en va finalement de même pour $\cL$.

En conclusion, en notant $\overline{X}$ le normalisé de $X'$, on obtient une $X_K$-compacti\-fi\-ca\-tion
$X\ra\overline{X}$ ayant les propriétés requises.
\end{proof}
  
Concernant les changements de traits, on a les résultats analogues à ceux de la section précédente.

\begin{Cor} \label{vauniv}
Soit $S^{\sh}\ra S$ une hensélisation stricte d'un trait $S$. Soient $n$ un entier $\geq 1$, et $T_i/S$, $i=1,\ldots,n$, des sous-extensions de traits fidèlement plates de $S^{\sh}/S$. Soit $X$ un $S$-schéma de type fini, séparé et plat, dont la fibre générique $X_K$ est un $K$-schéma propre géométriquement normal. On suppose le schéma $X$ régulier.

Il existe une $X_K$-compactification $X\ra\overline{X}$ de $X$, avec $\overline{X}$ plat sur $S$, \emph{normal} et tel que \emph{la flèche de restriction $\Pic(\overline{X}_{T})\ra\Pic(X_{T})$ soit surjective pour $T=T_i$, $i=1,\ldots,n$, et pour $T/S$ extension permise arbitraire.} 
\end{Cor}

\begin{proof}
Soit $X'/S$ une $X_K$-compactification de $X$, plate sur $S$. Quitte à modifier $X'$ par un éclatement $X$-admissible, on peut supposer que pour toute extension permise $T/S$, induisant une extension générique $L/K$, la flèche de restriction à la fibre générique
$$\Pic((X')_{T})\ra\Pic(X_{L})$$
est surjective (théorème \ref{universel} et remarque \ref{Remreg}). 

Soient $\cV_{k^s,j}$, $j=1,\ldots,\nu$, les idéaux inversibles sur $X_{S^{\sh}}$ définissant les composantes irréductibles réduites de la fibre spéciale de $X_{S^{\sh}}/S^{\sh}$.
Après un nouvel éclatement $X$-admissible de $X'$, on peut supposer que les $\cV_{k^s,j}$ se prolongent en des faisceaux inversibles sur $(X')_{S^{\sh}}$ (théorème \ref{va} appliqué à $S^{\sh}/S$). 

Alors, pour toute extension permise $T/S$, la flèche de restriction à $X_T$
$$\Pic((X')_{T})\ra\Pic(X_{T})$$
est surjective. En effet, le trait $T$ étant strictement hensélien, on peut fixer une factorisation $T\ra S^{\sh}\ra S$ de $T/S$. Comme le corps résiduel de $S^{\sh}$ est séparablement clos, les composantes irréductibles réduites de la fibre spéciale de $X_{T}/T$ sont définies par des idéaux inversibles provenant des $\cV_{k^s,j}$ par le changement de base $T/S^{\sh}$. Ils se prolongent donc en des faisceaux inversibles sur $(X')_T$. Comme $(X')_T$ est déjà semi-factoriel, et $X_T$ est régulier (remarque \ref{normal}), l'homomorphisme de restriction ci-dessus est bien surjectif.

Ceci étant, après $n$ éclatements $X$-admissibles, on peut en outre assurer que les restrictions
$$\Pic((X')_{T_i})\ra\Pic(X_{T_i})$$
soient surjectives pour $i=1,\ldots,n$ (théorème \ref{va} appliqué à $T_i/S$). Le normalisé $\overline{X}$ de $X'$ est alors une $X_K$-compactification convenable de $X$.
\end{proof}

Le corollaire \ref{vauniv} s'applique par exemple pour obtenir des \emph{compactifications semi-factorielles du modèle de Néron d'une variété abélienne}. Plus généralement, rappelons que si $S$ est un trait de corps de fonctions $K$, un torseur $\fppf$ $X_K$ sous une variété abéleinne $A_K$ est dit \emph{non ramifié} si, notant $K^{\sh}$ le corps de fonctions d'un hensélisé strict de $S$, l'ensemble $X_K(K^{\sh})$ est non vide.

\begin{Cor} \label{mn}
Soit $S$ un trait de corps de fonctions $K$. Soient $A_K$ une variété abélienne, et $X_K$ un torseur $\fppf$ non ramifié sous $A_K$. Soit $X$ le modèle de N\'eron de $X_K$ sur $S$.

Il existe une $X_K$-compactification $X\ra\overline{X}$ de $X$, avec $\overline{X}$ projectif et plat sur $S$, \emph{normal, tel que la flèche de restriction $\Pic(\overline{X})\ra\Pic(X)$ soit surjective, et le soit encore après toute extension permise $T/S$ (définition \ref{permis}).} 

Si on se prescrit un nombre fini d'extensions de traits fidèlement plates $T_i/S$, sous-extensions d'une hensélisation stricte $S^{\sh}/S$, on peut en outre demander que $\Pic(\overline{X}_{T_i})\ra\Pic(X_{T_i})$ soit surjective pour tout $i$.
\end{Cor}

\begin{proof}
Le schéma $X_K$ possède bien un modèle de Néron $X$ sur $S$ d'après \cite{BLR} 6.5/4. Celui-ci est quasi-projectif sur $S$ (\emph{loc. cit.} 6.4/1). D'après le corollaire \ref{vauniv}, on peut donc bien trouver un $\overline{X}$ projectif et plat sur $S$ vérifiant les conclusions du théorème (dans la démonstration, il suffit de commencer avec une compactification \emph{projective} de $X/S$).
\end{proof}
L'intérêt de telles compactifications est d'exister \emph{sans aucune hypothèse sur la réduction de la variété abélienne $A_K$.} Bien entendu, lorsque $A_K$ a réduction semi-stable sur $S$, on dispose de compactifications bien meilleures de son modèle de Néron $A/S$ : en s'appuyant sur la théorie de la dégénérescence des variétés abéliennes de Mumford-Faltings-Chaï, Künnemann a construit une compactification régulière canonique de $A/S$ (cf. \cite{Kü}).

\subsection{Une variante globale de la semi-factorialité}\label{global}

\begin{Th} 
Soit $S$ un schéma noethérien intègre normal de dimension $1$, de corps de fonctions $K$. Soient $n$ un entier $\geq 1$, et $T_i/S$, $i=1,\ldots,n$, des $S$-schémas étales de type fini, avec $T_i$ intègre de corps de fonctions $L_i$. Soit $X$ un $S$-schéma propre et plat. On suppose que la fibre générique $X_K$ de $X/S$ est \emph{lisse} sur $K$.

Il existe un $S$-schéma $\widetilde{X}$ isomorphe à $X$ au dessus d'un ouvert non vide de $S$, ayant les propriétés suivantes : $\widetilde{X}/S$ est propre et plat, \emph{$\widetilde{X}$ est normal}, et \emph{la flèche de restriction $\Pic(\widetilde{X}_{T_i})\ra\Pic(X_{L_i})$ est surjective pour tout $i=1,\ldots,n$.}
\end{Th}

\begin{proof}
Comme $X_K$ est lisse, il existe un ensemble fini $F:=\{s_1,\ldots,s_m\}$ de points fermés de $S$ tel que $X_U:=X\times_S U$ soit lisse sur l'ouvert $U:=S-F$ de $S$. Soit $j\in\{1,\ldots,m\}$. D'après le corollaire \ref{Ts}, il existe un modèle propre plat normal $X_{(s_j)}$ de $X_K$ sur $\cO_{S,s_j}$, tel que pour tout $i=1,\ldots,n$, et pour tout point $t_{i_j}$ de $T_i$ au-dessus de $s_j$, le schéma $X_{(s_j)}\otimes_{\cO_{S,s_j}}\cO_{T_i,t_{i_j}}$ soit semi-factoriel sur $\cO_{T_i,t_{i_j}}$. Chacun des $X_{(s_j)}$, $j=1,\ldots,m$, peut être étendu en un schéma $X_j$ de type fini au-dessus d'un voisinage ouvert $S_j$ de $s_j$ dans $S$ (\cite{EGA IV}${}_3$ 8.8.2). Quitte à diminuer $S_j$ autour de $s_j$, on peut supposer $X_j/S_j$ propre (\emph{loc. cit.} 8.10.5 (xii)), puis plat. On peut même supposer $S_j\cap F=\{s_j\}$ et $X_j=X$ au-dessus de $S_j\cap U$. Chaque $X_j$ peut alors être recollé à $X_U$ au-dessus de $S_j\cap U$, pour obtenir un schéma  $\widetilde{X}$ propre et plat sur $S$, lisse au-dessus de $U$, et normal.

Montrons que la restriction $\Pic(\widetilde{X}_{T_i})\ra\Pic(X_{L_i})$ est surjective pour tout $i=1,\ldots,n$. Comme la formation de $\widetilde{X}$ commute au changement de base $T_i\ra S$, on peut supposer $T_i=S$, et il faut alors montrer que $\Pic(\widetilde{X})\ra\Pic(X_K)$ est surjective. Soit $\cL_K$ un faisceau inversible sur $X_K$. Comme $X_K$ est réduit, on peut supposer que $\cL_K$ est le faisceau défini par un diviseur $D_K$ sur $X_K$. On notera $H$ le cycle $1$-codimensionnel sur $\widetilde{X}$, obtenu en prenant l'adhérence schématique du cycle $\cyc(D_K)$ défini par $D_K$ sur $X_K$. Par ailleurs, il existe, pour tout $j=1,\ldots,m$, un diviseur $D_j$ qui prolonge $D_K$ sur $X_{(s_j)}$. La partie verticale $V_j$ du cycle $\cyc(D_j)$ est une combinaison $\bbZ$-linéaire des composantes irréductibles réduites de la fibre $X_{s_j}$. On peut donc voir $V_j$ comme un cycle $1$-codimensionnel sur $\widetilde{X}$. On va montrer que 
$$Z:=H+\sum_{j=1}^n V_j$$
est de la forme $\cyc(D)$ pour un certain diviseur $D$ sur $\widetilde{X}$, ce qui achèvera la démonstration puisque $X_K$ est normal. 

Comme le schéma $\widetilde{X}$ est normal, il suffit de voir que le cycle $Z$ est localement principal. Ecrivons $Z=\sum_{\xi\in \widetilde{X}^{(1)}}n_\xi\cdot\overline{\{\xi\}}$ et posons pour tout $x\in \widetilde{X}$,
$$T_x:=\Spec(\cO_{\widetilde{X},x}),\quad Z_x:=\sum_{\xi\in \widetilde{X}^{(1)}}n_\xi\cdot(\overline{\{\xi\}}\cap T_{x}).$$
Le cycle $Z$ est localement principal si et seulement si, pour tout $x\in \widetilde{X}$, il existe une fonction rationnelle $f$ sur $T_x$ telle $Z_x=\cyc(\div(f))$ (\cite{EGA IV}${}_4$ 21.6.7). Maintenant, si $x\in X$ se projette sur un point $s_j$ de $F$, le cycle $Z_x$ coïncide avec le cycle
$$\sum_{\xi\in (X_{(s_j)})^{(1)}}n_\xi\cdot(\overline{\{\xi\}}\cap T_{x})\quad=\quad(H+V_j)\cap T_x,$$
qui n'est autre que le cycle localement principal $\cyc(D_j|_{T_x}).$ Sinon, le point $x$ appartient a l'ouvert régulier $\widetilde{X}_U$, sur lequel $Z$ est automatiquement localement principal. Le cycle $Z$ est donc bien localement principal sur $\widetilde{X}$.
\end{proof}

\section{Modèle de Néron d'une variété de Picard} \label{section2}

Soit $S$ un trait de corps de fonctions $K$ et corps résiduel $k$ au point fermé. Soit $X_K$ un $K$-schéma propre géométriquement normal. Sa variété de Picard $A_K:=\Pic_{X_K/K,\red}^0$ est une variété abélienne sur $K$ (\cite{G} 3.2). En particulier, elle possède un modèle de Néron $A$ de type fini sur $S$. Soit d'autre part $X/S$ un modèle propre et plat de $X_K$, qui est semi-factoriel sur un hensélisé strict de $S$. Un tel modèle existe toujours d'après le théorème \ref{T}. Le foncteur de Picard $\Pic_{X/S}$ est le faisceau $\fppf$ associé au préfaisceau $T\ra\Pic(X\times_S T)$. Dans cette section, on décrit $A$ en termes de $\Pic_{X/S}$ (théorèmes \ref{Néron} et \ref{Néronbis}). On en tire ensuite une conséquence sur l'équivalence algébrique des faisceaux inversibles sur $X$ (corollaire \ref{n}).

\subsection{Foncteur de Picard et modèle de Néron}
Commençons avec un trait $S$ de corps de fonctions $K$, et un schéma $X/S$ propre et plat à fibre générique $X_K$ géométriquement normale, de sorte $A_K:=\Pic_{X_K/K,\red}^0$ est un $K$-schéma \emph{en groupes}, et même une \emph{variété abélienne}.

Soit $E$ l'adhérence schématique de la section unité de $\Pic_{X_K/K}$ dans $\Pic_{X/S}$ (\cite{R} 3.2 c)). Le quotient $\fppf$ $\Pic_{X/S}/E$ a pour fibre générique $\Pic_{X_K/K}$. On note $Q$ l'adhérence schématique de $A_K$ dans $\Pic_{X/S}/E$, et $P$ l'image réciproque de $Q$ par l'épimorphisme canonique $q:\Pic_{X/S}\twoheadrightarrow\Pic_{X/S}/E$. L'inclusion $Q\hookrightarrow \Pic_{X/S}/E$ induit donc un morphisme cartésien de suites exactes :
\begin{displaymath}
\xymatrix{
0 \ar[r]  & E \ar[r] \ar@{=}[d] & \Pic_{X/S} \ar[r]^{q}        & \Pic_{X/S}/E \ar[r] & 0 \\
0 \ar[r]            & E \ar[r] & P \ar[r] \ar@{^{(}->}[u] & Q \ar[r] \ar@{^{(}->}[u] & 0.
}
\end{displaymath}
Choisissons un rigidificateur $Y\ra X$ de $\Pic_{X/S}$ (\cite{R} 2.2.3 c)). On peut considérer le foncteur de Picard de $X/S$ relatif au rigidificateur $Y$, noté $(\Pic_{X/S},Y)$ (\emph{loc. cit.} 2.1). C'est un espace algébrique en groupes localement de type fini sur $S$ (\emph{loc. cit.} 2.3.1), et on dispose d'un épimorphisme canonique $r:(\Pic_{X/S},Y)\twoheadrightarrow\Pic_{X/S}$
(\emph{loc. cit.} 2.1.2). Notant $H$ l'adhérence schématique du noyau de $r_K$ dans $(\Pic_{X/S},Y)$, on a $r(H)\subseteq E$ par continuité. Il est prouvé dans \emph{loc. cit.} 4.1 que 
$$(\Pic_{X/S},Y)/H \xrightarrow{\bar{r}} \Pic_{X/S}/E$$ 
est un isomorphisme, et que par suite $\Pic_{X/S}/E$ est un schéma en groupes localement de type fini et séparé sur $S$. Par suite, l'inclusion $Q\hookrightarrow \Pic_{X/S}/E$ est une immersion fermée de schémas et $Q$ est plat sur $S$.

Soit $(P,Y)$ l'image réciproque de $P$ par $r$, i.e. l'image réciproque de $Q$ par le composé $q\circ r$. C'est un sous-espace fermé de $(\Pic_{X/S},Y)$, qui par conséquent contient $H$. D'où un carré cartésien
\begin{displaymath}
\xymatrix{
(\Pic_{X/S},Y)/H \ar[r]^{\bar{r}} & \Pic_{X/S}/E \\
(P,Y)/H \ar[r]^{\bar{\bar{r}}} \ar@{^{(}->}[u]& P/E=Q. \ar@{^{(}->}[u]
}
\end{displaymath}
En particulier $\bar{\bar{r}}$ est un isomorphisme. 

En appliquant \cite{SGA3} I, Exp. VI${}_B$, 9.2 (xi) (ou plutôt son analogue pour les $S$-espaces algébriques en groupes) à la suite exacte 
$$0\ra H \ra (P,Y) \ra Q \ra 0,$$
on obtient que $(P,Y)$ est $S$-plat. Celui-ci est donc contenu dans $\overline{r^{-1}(A_K)}$, adhérence schématique dans $(\Pic_{X/S},Y)$ de l'image réciproque de $A_K$ par $r$. Mais par continuité, on a $q\circ r \big(\overline{r^{-1}(A_K)}\big)\subseteq Q$, i.e. $\overline{r^{-1}(A_K)}\subseteq (P,Y)$. Ainsi $(P,Y)=\overline{r^{-1}(A_K)}$. Puis  
$$P=r\big((P,Y)\big)=r\big(\overline{r^{-1}(A_K)}\big)$$
est contenu (par continuité) dans $\overline{A_K}$, adhérence schématique de $A_K$ dans $\Pic_{X/S}$. Comme $q(\overline{A_K})\subseteq Q$ (toujours par continuité), on trouve finalement $P=\overline{A_K}$. En résumé :
\vspace*{2mm}
\begin{itemize}
\item $X$ schéma propre et plat sur $S$ à fibre générique $X_K$ géométriquement normale
\item $A_K:=\Pic_{X_K/K,\red}^0$ variété de Picard de $X_K$
\item $\Pic_{X/S}$ foncteur de Picard $\fppf$ de $X/S$
\item $P$ adhérence schématique de $A_K$ dans $\Pic_{X/S}$
\item $E$ adhérence schématique de $0_K\in A_K$ dans $\Pic_{X/S}$
\item $Q:=P/E$ plus grand quotient séparé sur $S$ de $P$, représentable par un schéma localement de type fini et plat sur $S$
\item $(\Pic_{X/S},Y)$ foncteur de Picard de $X/S$ relatif au rigidificateur $Y$, muni de la flèche d'oubli $r:(\Pic_{X/S},Y)\ra\Pic_{X/S}$
\item $(P,Y)$ adhérence schématique de $r^{-1}(A_K)$ dans $(\Pic_{X/S},Y)$
\item $H$ adhérence schématique de $r^{-1}(0_K)$ dans $(\Pic_{X/S},Y)$.
\end{itemize}
\vspace*{2mm}
Lorsque $X_K$ est une courbe, ces notations sont les mêmes que celles utilisées dans \cite{BLR} 9.5, en notant que dans ce cas, $\Pic_{X_K/K,\red}^0=\Pic_{X_K/K}^0$. Par ailleurs, avec ces notations, la fibre générique de $P$ est $P_K=Q_K=A_K=\Pic_{X_K/K,\red}^0$, à ne pas confondre avec la notation $P_K:=\Pic_{X_K/K}$ utilisée section \ref{section1}, et qui n'a plus d'utilité ici.

\begin{Lem} \label{extension} Supposons de plus $X_K$ géométriquement connexe. Si $X_K(K)$ est non vide, ou si $S$ est hensélien à corps résiduel algébriquement clos, alors on a l'égalité 
$$\Pic(X_K)=\Pic_{X/K}(K).$$
Dans ce cas, si $X$ est \emph{semi-factoriel sur $S$}, alors le morphisme de restriction 
$$Q(S)\ra Q(K)$$ 
est bijectif.
\end{Lem}

\begin{proof} 
Comme $X_K$ est propre géométriquement intègre, on a $\cO(X_K)=K$, et la suite spectrale de Leray pour $\bbG_{mK}$ induit une suite exacte
$$0\ra\Pic(X_K)\ra\Pic_{X_K/K}(K)\ra\Br(K)\ra\Br(X_K),$$
où $\Br$ désigne le groupe de Brauer (\cite{BLR} 8.1/4). Si $S$ est hensélien à corps résiduel algébriquement clos, c'est un théorème de Lang que $\Br(K)=0$ (voir \cite{BLR} page 203 pour des références). Donc, dans les deux cas envisagés dans l'énoncé, la flèche $\Br(K)\ra\Br(X_K)$ est injective et $\Pic(X_K)=\Pic_{X_K/K}(K)$.

L'injectivité de $Q(S)\ra Q(K)$ est déjà acquise puisque $Q$ est séparé sur $S$. Montrons la surjectivité, sous l'hypothèse que $X/S$ est semi-factoriel.
Soit $a_K\in Q(K)=\Pic_{X_K/K}^0(K)$. D'après ce qui précède, on peut représenter $a_K$ par un faisceau inversible $\cL_K$ sur $X_K$. Puisque $X$ est semi-factoriel sur $S$, il existe un $\cO_X$-module inversible $\cL$ prolongeant $\cL_K$. L'image de la classe de $\cL$ par la composée
$$\Pic(X)\ra\Pic_{X/S}(S)\ra(\Pic_{X/S}/E)(S)$$ prolonge $a_K$ dans $Q(S)$.
\end{proof}

On peut maintenant énoncer l'analogue, en dimension supérieure, du théorème \cite{R} 8.1.4 a) (ou \cite{BLR} 9.5/4) de Raynaud.

\begin{Th} \label{Néron}
Soit $S$ un trait de corps de fonctions $K$ et corps résiduel $k$ au point fermé. Soit $X_K$ un $K$-schéma propre, géométriquement normal et géométriquement connexe. Soit $X/S$ un modèle propre et plat de $X_K$, semi-factoriel sur un hensélisé strict de $S$. Supposons que $X$ possède une quasi-section étale, ou que $k$ soit parfait. Alors le modèle de Néron $A$ de $A_K:=\Pic_{X_K/K,\red}^0$ sur $S$ est le séparé lissifié de l'adhérence schématique de $A_K$ dans $\Pic_{X/S}$. 

Plus précisément, si $P$ est cette adhérence schématique, et $E$ celle de la section unité de $\Pic_{X_K/K}$, le quotient $\fppf$ $Q:=P/E$ est un schéma en groupes localement de type fini et séparé sur $S$. Notant $Q'$ son lissifié, l'identité de $A_K$ se prolonge en un isomorphisme de $S$-schémas
\begin{displaymath}
\xymatrix{
Q' \ar[r]^{\sim} & A.
}
\end{displaymath}
\end{Th}

\begin{proof}
Toutes les données de l'énoncé commutent à l'hensélisation stricte de $S$. Pour montrer que la flèche canonique $Q'\ra A$ est un isomorphisme, on peut donc supposer $S$ strictement hensélien.

Le $S$-schéma en groupes $Q'$ est séparé et lisse, de fibre générique $A_K$. De plus, le morphisme 
$$Q'(S)\ra Q'(K)$$ 
est bijectif d'après le lemme \ref{extension}. Pour démontrer le théorème, il reste à vérifier que $Q'$ est de type fini sur $S$ (cf. \cite{BLR} 7.1/1). Pour cela, la méthode est la même que dans \cite{BLR} 9.5/7, à ceci près qu'il faut travailler avec $Q'$ au lieu de $Q$. Comme dans \emph{loc. cit.}, il suffit de montrer que le morphisme canonique $Q'\ra A$ induit un isomorphisme sur les composantes neutres.

Pour cela, le point clé à démontrer est que le groupe abstrait des composantes connexes $\Phi_{Q'}:=Q'(k)/(Q')^0(k)$ de $(Q')_k$ est de type fini. Le même raisonnement que celui de \emph{loc. cit.} permet alors de conclure que $(Q')^0\ra A^0$ est bien un isomorphisme.

Dans la suite de la démonstration, si $G$ est un $S$-foncteur en groupes dont la fibre spéciale $G_k$ est représentable par un $k$-schéma localement de type fini, on note $\Phi_G$ le groupe abstrait des composantes connexes de $G_k$, à savoir
$$\Phi_{G}:=G_{k}(\bar{k})/G_{k}^{0}(\bar{k})$$
où $\bar{k}$ est une clôture algébrique de $k$. On a 
$$\Phi_G=G_{k}(k)/G_{k}^{0}(k)$$
lorsque $G_k$ est lisse sur le corps séparablement clos $k$.

Avec ces notations, la lissification $Q'\ra Q$ induit un morphisme $\Phi_{Q'}\ra\Phi_Q$. Si $\Phi_{Q}'$ désigne son image et $N$ son noyau, on a une suite exacte de groupes abstraits
$$0\ra N\ra\Phi_{Q'}\ra\Phi_{Q}'\ra 0.$$ Comme la lissificaion est un morphisme de type fini, le groupe $N$ est fini. On est donc ramené à montrer que $\Phi_Q$ est de type fini.

La fibre spéciale $P_k$ de $P$ est représentable par un sous-$k$-schéma en groupes fermé de $\Pic_{X_k/k}$ (à savoir $q_{k}^{-1}(Q_k)$). Le groupe $\Phi_Q$ est quotient du groupe $\Phi_P$. Notant $\Phi_{\Pic_{X/S}}'$ l'image de $\Phi_P$ dans le groupe $\Phi_{\Pic_{X/S}}$, on a une extension
$$0\ra M\ra\Phi_{P}\ra\Phi_{\Pic_{X/S}}'\ra 0,$$
avec $M$ fini. Maintenant, le groupe de Néron-Séveri géométrique $\Phi_{\Pic_{X/S}}$ de $X_k$ est de type fini (\cite{SGA6} XIII 5.1). Par conséquent $\Phi_P$, et donc $\Phi_Q$, sont de type fini. Ceci achève la démonstration.
\end{proof}

On a réalisé le modèle de Néron $A$ de $A_K$ en séparant le foncteur $P$, puis en lissifiant le schéma en groupes obtenu. Le processus inverse, à savoir lissification \emph{puis} séparation, permet aussi de retrouver $A$. 

Plus précisément, la dilatation d'un sous-schéma fermé dans un $S$-schéma localement de type fini est une construction locale pour la topologie étale, et c'est un morphisme affine, comme il résulte directement des définitions (cf. \cite{BLR} 3.2). Les dilatations s'étendent donc à la catégorie des $S$-espaces algébriques localement de type fini. En particulier, le procédé de  lissification des groupes décrit dans \cite{BLR} page 174 s'étend à la catégorie des $S$-espaces algébriques en groupes localement de type fini à fibre générique lisse. (On notera que, dans l'énoncé \cite{BLR} 7.1/4, on peut remplacer l'hypothèse << de type fini >> par << localement de type fini >> : le défaut de lissité d'une $S^{\sh}$-section quelconque est égal à celui de la section neutre, ce qui permet, pour un $S$-groupe, de s'affranchir de l'hypothèse de quasi-compacité dans \emph{loc. cit.} 3.3/3.)

Ceci étant, notons comme précédemment $(P,Y)$ l'adhérence schématique de $r^{-1}(A_K)$ dans $(\Pic_{X/S},Y)$, où $r$ désigne l'épimorphisme canonique
$$(\Pic_{X/S},Y)\twoheadrightarrow\Pic_{X/S}.$$
C'est un $S$-espace algébrique en groupes localement de type fini, dont la fibre générique $r^{-1}(A_K)$ est lisse (\cite{R} 2.4.3 b)). Posons :
\vspace*{2mm}
\begin{itemize}
\item $(P,Y)'$ espace algébrique en groupes sur $S$, lissifié de $(P,Y)$
\item $I$ adhérence schématique de $r^{-1}(0_K)$ dans $(P,Y)'$.
\end{itemize}
\vspace*{2mm}
 
\begin{Th} \label{Néronbis}
Soit $S$ un trait de corps de fonctions $K$ et corps résiduel $k$ au point fermé. Soit $X_K$ un $K$-schéma propre, géométriquement normal et géométriquement connexe. Soit $X/S$ un modèle propre et plat de $X_K$, semi-factoriel sur un hensélisé strict de $S$. Supposons que $X$ possède une quasi-section étale, ou que $k$ soit parfait. Alors le modèle de Néron $A$ de $A_K:=\Pic_{X_K/K,\red}^0$ sur $S$ admet une présentation (pour la topologie $\fppf$) par des $S$-espaces algébriques en groupes
$$0\ra I\ra (P,Y)'\ra A \ra 0,$$
où $(P,Y)'\ra A$ est l'unique $S$-morphisme prolongeant $r_K:(P,Y)_K\ra A_K$.

Si $X$ est cohomologiquement plat sur $S$, le foncteur $\Pic_{X/S}$ est un espace algébrique (\cite{R} 5.2). Notant alors $F$ l'adhérence schématique de $0_K\in A_K$ dans le lissifié $P'$ de $P$, l'identité de $A_K$ induit la suite exacte
$$0\ra F\ra P'\ra A\ra 0.$$
\end{Th}
Les ingrédients de la démonstration sont essentiellement les mêmes que pour le théorème \ref{Néron}, mais la preuve en est indépendante.

\begin{proof}
On a vu en début de section que le $S$-schéma en groupes $Q$ admet une présentation par des $S$-espaces algébriques en groupes
$$0\ra H\ra (P,Y)\ra P/E=Q \ra 0.$$
Par continuité, la lissication $(P,Y)'\ra (P,Y)$ induit une flèche $I\ra H$. Notant $C$ le faisceau $\fppf$ quotient $(P,Y)'/I$, on obtient un morphisme de suite exactes
\begin{displaymath}
\xymatrix{
0\ar[r] & I\ar[r] \ar[d] & (P,Y)' \ar[r] \ar[d] & C \ar[r] \ar[d] & 0 \\
0\ar[r] & H \ar[r] & (P,Y) \ar[r] & Q \ar[r] & 0.
}
\end{displaymath}
Comme $I$ est plat sur $S$, le faisceau $C$ est encore un $S$-espace algébrique en groupes (cf. par exemple \cite{BLR} 8.4/9). Il est localement de type fini sur $S$ d'après l'analogue de \cite{SGA3} I, Exp. VI${}_B$, 9.2 (xii) pour les espaces algébriques en groupes. Comme il est séparé sur $S$ ($I$ est fermé dans $(P,Y)'$), c'est un $S$-schéma (cf. \cite{A} IV 4.B). Il est lisse sur $S$ puisqu'il est quotient de $(P,Y)'$ qui lisse sur $S$ par $I$ qui est plat sur $S$ (\cite{SGA3} \emph{loc. cit.}). En particulier, l'identité de $A_K$ se prolonge en un $S$-morphisme $C\ra A$. Pour montrer que c'est un isomorphisme, on peut supposer $S$ strictement hensélien. Il reste alors à voir que $C(S)\ra C(K)$ est surjectif et que $C$ est de type fini sur $S$.

Soit $a_K\in C(K)=A_K(K)=\Pic_{X/K}^0(K)$. D'après le lemme \ref{extension}, on peut représenter $a_K$ par un faisceau inversible $\cL_K$ sur $X_K$. Comme $X$ est semi-factoriel sur $S$, on peut prolonger $\cL_K$ en un faisceau inversible $\cL$ sur $X$. Par définition, le rigidificateur $Y$ est fini sur $S$. On peut donc rigidifier $\cL$ en $(\cL,\alpha)$, définissant ainsi une section $b\in(P,Y)(S)$. Celle-ci se relève en une section $b'$ du lissifié $(P,Y)'$, et par commutativité du diagramme ci-dessus, l'image $a\in C(S)$ de $b'$ a pour fibre générique $a_K$. D'où la surjectivité de $C(S)\ra C(K)$.

Montrons enfin que $C/S$ est de type fini. Comme dans \ref{Néron} (et dans \cite{BLR} 9.5/7), cela revient à montrer que le groupe ordinaire $\Phi_C$ des composantes connexes de $C_k$ est de type fini.

Le groupe $\Phi_C$ est quotient du groupe $\Phi_{(P,Y)'}$, et le morphisme de $k$-schémas 
$$\big((P,Y)'\big)_k\ra (P,Y)_k$$ 
est de type fini. On est donc ramené à voir que $\Phi_{(P,Y)}$ est de type fini. Mais le noyau de
$$(P,Y)_k\ra P_k$$
est connexe (\cite{R} 2.4.3 b)), si bien que $\Phi_{(P,Y)}=\Phi_P$. Ce dernier est de type fini parce que le groupe de Néron-Séveri géométrique $\Phi_{\Pic_{X/S}}$ l'est. Ceci achève la démonstration de la première partie du théorème.

Maintenant, lorsque $\Pic_{X/S}$ est un espace algébrique, il n'y a pas lieu de considérer le rigidificateur $Y$, et la seconde partie du théorème se démontre comme la première.
\end{proof}

\begin{Rem} \label{com}
\emph{Avec les notations des théorèmes \ref{Néron} et \ref{Néronbis}, la flèche composée}
$$(P,Y)'\ra (P,Y)\ra P \ra Q$$ \emph{se factorise de façon unique à travers la lissification $Q'\ra Q$.}
\emph{Le diagramme}
\begin{displaymath}
\xymatrix{
(P,Y)'\ar[rr] \ar[dr] & &  Q' \ar[dl]^{\sim} \\
 & A &
}
\end{displaymath}
\emph{est commutatif. En effet, les deux morphismes $(P,Y)'\ra A$ coincïdent avec $r_K$ sur les fibres génériques, l'espace $(P,Y)'$ est lisse sur $S$ (donc réduit) et $A/S$ est séparé.}
\end{Rem}

\subsection{Comparaison des composantes neutres}

Rappelons la définition suivante.

\begin{Def}(\cite{R} 6.1.11 3)) \label{d'}
Soit $S$ un trait de corps résiduel $k$ au point fermé. Soit $X$ est un schéma \emph{normal}, propre et plat sur $S$. On note $\xi_i$, $i=1,\ldots,r$, les points maximaux de la fibre spéciale $X_k$ de $X/S$, $d_i$ la longueur de l'anneau artinien $\cO_{X_k,\xi_i}$, et $X_i$ la composante irréductible de $X_k$ de point générique $\xi_i$ munie de sa structure réduite.

On définit un entier $d'$ par la condition suivante : c'est le plus grand entier divisant les $d_i$, $i=1,\ldots,r$, tel que le cycle $1$-codimensionnel sur $X$
$$(1/d')(X_k)=\sum_{i=1}^{r}(d_i/d')(X_i)$$
soit un diviseur.
\end{Def}

\begin{Prop} \label{neutre}
Dans la situation du théorème \ref{Néronbis}, la flèche $r_K$ se prolonge en un morphisme fidèlement plat localement de type fini
$$\big((P,Y)'\big)^0\twoheadrightarrow A^0.$$
Supposons $X$ cohomologiquement plat sur $S$. Si $X$ est normal et $d'=1$ (définition \ref{d'}), alors l'identité de $A_K$ se prolonge en un isomorphisme
\begin{displaymath}
\xymatrix{
(P')^0\ar[r]^{\sim} & A^0.
}
\end{displaymath}
\end{Prop}

\begin{proof}
Comme $I/S$ est fidèlement plat localement de type fini, l'épimorphisme  
$$(P,Y)'\twoheadrightarrow A$$
l'est aussi. La fibre générique $r_{K}^{-1}(A_K)$ de $(P,Y)'$ est connexe (puisque $A_K$ et $\Ker(r_K)$ le sont, cf \cite{R} 2.4.3 b)). L'inclusion $\big((P,Y)'\big)^0\subseteq (P,Y)'$ est donc ouverte, et par conséquent le morphisme induit sur les composantes neutres
$$\big((P,Y)'\big)^0\ra A^0$$
est localement de type fini, plat et surjectif sur les fibres génériques. En particulier, l'image du $k$-morphisme
$$\big((P,Y)_{k}'\big)^0\ra A_{k}^0$$
est ouverte, et donc égale au groupe $A_{k}^0$ tout entier. Par conséquent, le morphisme $\big((P,Y)'\big)^0\ra A^0$ est surjectif.

Lorsqu'on suppose de plus $X$ cohomologiquement plat sur $S$, le résultat ci-dessus s'applique au morphisme $(P')^0\twoheadrightarrow A^0.$ Pour voir que c'est un isomorphisme si $X$ est normal et << $d'=1$ >>, on peut supposer $S$ strictement hensélien. L'adhérence schématique $F$ de la section unité de $A_{K}$ dans le $S$-espace algébrique localement séparé $P'$ est étale sur $S$ (\cite{R} 3.3.5), et le noyau de la flèche précédente est l'ouvert $F\cap (P')^0$ de $F$. Mais on a $(F\cap (P')^0)(S)\subseteq (E\cap P^0)(S)$, et les hypothèses entraînent que $(E\cap P^0)(S)=0$ (\cite{R} 6.4.1 3)). D'où l'isomorphisme annoncé.
\end{proof}
Pour déduire de la proposition précedente un énoncé de surjectivité au niveau des \emph{sections}, on va se placer dans le cas où $\cO(S)$ est complet à corps résiduel algébriquement clos. C'est alors un argument utilisant le \emph{foncteur de Greenberg} d'un espace algébrique en groupes commutatifs lisse sur $S$ qui permet de conclure. Cela n'apparaîtra pas explicitement ici, mais intervient crucialement dans la démonstration de \cite{BLR} 9.6/2 que nous allons utiliser.

\begin{Prop} \label{surj}
On conserve les notations du théorème \ref{Néronbis}, et on suppose de plus \emph{$\cO(S)$ complet et $k$ algébriquement clos.} Le morphisme canonique
$$\big((P,Y)'\big)^0(S)\ra A^0(S)$$ est surjectif.
\end{Prop}

\begin{proof}
D'après \cite{BLR} 9.6/2, il suffit de vérifier que le groupe abstrait $(P,Y)'(S)/\big((P,Y)'\big)^0(S)$ est de type fini et que le morphisme $(P,Y)'(S)\ra A(S)$ est surjectif.

L'espace algébrique $(P,Y)'$ est lisse sur $S$, et sa fibre générique $r_{K}^{-1}(A_K)$ est connexe puisque $A_K$ et $\Ker(r_K)$ le sont (\cite{R} 2.4.3 b)). Le morphisme de groupes
$(P,Y)'(S)/\big((P,Y)'\big)^0(S)\ra \Phi_{(P,Y)'}$
est donc bijectif. Et on a déjà vu au cours de la démonstration du théorème \ref{Néronbis} que ce groupe est de type fini.

Passons à la surjectivité de $(P,Y)'(S)\ra A(S)$. Soit $a\in A(S)$. Notons $a_K\in A_K(K)$ sa fibre générique. On a vu au cours de la démonstration du théorème \ref{Néronbis} que l'on peut trouver $b'\in (P,Y)'(S)$ dont l'image dans $A(S)$ prolonge $a_K$. Comme $A/S$ est séparé, l'image de $b'$ est $a$. 
\end{proof}

\begin{Cor} \label{se}
Dans la situation de la proposition \ref{surj}, on a un isomorphisme
\begin{displaymath}
\xymatrix{
\big((P,Y)'\big)(S)\Big\slash\Big(\big((P,Y)'\big)^0(S)+I(S)\Big)\ar[r]^>>>>>{\sim} & A(S)/A^0(S)
}
\end{displaymath}
c'est-à-dire, en termes des groupes des composantes connexes des fibres spéciales, une suite exacte
\begin{displaymath}
\xymatrix{
0\ar[r] & I(S)\Big\slash\Big(I(S)\cap\big((P,Y)'\big)^0(S)\Big)\ar[r] & \Phi_{(P,Y)'}\ar[r]& \Phi_A \ar[r] & 0.
}
\end{displaymath}
\end{Cor}

\begin{Rem}
\emph{Lorsque de plus $X/S$ est cohomologiquement plat, i.e. $\Pic_{X/S}$ est un espace algébrique, le corollaire \ref{se} s'écrit}
\begin{displaymath}
\xymatrix{
P'(S)\big\slash\big((P')^0(S)+F(S)\big)\ar[r]^>>>>>{\sim} & \Phi_A.
}
\end{displaymath}
\emph{(avec les notations du théorème \ref{Néronbis}). La bijection $P'(S)\simeq P(S)$ induit une bijection $F(S)\simeq E(S)$. Le groupe $\Phi_A$ s'insère donc dans une suite exacte courte}
\begin{displaymath}
0 \ra P^0(S)\Big\slash\Big(P^0(S)\cap\big((P')^0(S)+F(S)\big)\Big) \ra \Phi_A \\
  \ra P(S)\big\slash\big(P^0(S)+E(S)\big)\ra 0.
\end{displaymath}
\emph{Si $X$ est normal et  $d'=1$ (définition \ref{d'}), le noyau se simplifie en $P^0(S)/(P')^0(S)$ (puisqu'alors $E(S)\cap P^0(S)=0$ d'après \cite{R} 6.4.1 3)).}
\end{Rem}

Terminons par un corollaire de la proposition \ref{surj} concernant l'équivalence algébrique sur $X$.

\begin{Def} \label{defalg} \emph{(cf. \cite{SGA6} XIII 4.1 et 4.4)}
Soit $Z\ra T$ un morphisme propre de schémas. Soit $\cL$ un $\cO_{Z}$-module inversible. On dit que $\cL$ est \emph{algébriquement équivalent à zéro sur $Z$ (relativement à $T$)} si son image dans $\Pic_{Z/T}(T)$ appartient au sous-groupe $\Pic_{Z/T}^0(T)$. On notera $\Pic^0(Z/T)$ (ou simplement $\Pic^0(Z)$ s'il n'y a pas d'ambiguïté) le sous groupe du groupe de Picard de $Z$ formé des classes de $\cO_Z$-modules inversibles algébriquement équivalents à zéro sur $Z/T$.

Par définition de $\Pic_{Z/T}^0$, le fait pour un $\cO_Z$-module inversible d'être algébriquement équivalent à zéro relativement à $T$ peut se vérifier sur les fibres de $Z/T$.
\end{Def}

\begin{Cor} \label{n}
Soit $S$ le spectre d'un anneau de valuation discrète complet, à corps résiduel algébriquement clos $k$ et corps des fractions $K$. Soit $X_K$ un $K$-schéma propre, géométriquement normal et géométriquement connexe. Soit $X/S$ un modèle de $X_K$, propre, plat et semi-factoriel sur $S$. Soient $A/S$ le modèle de Néron de la variété de Picard de $X_K$ et $n$ l'exposant du groupe des composantes connexes de $A_k$. 

Pour tout $\cO_{X_K}$-module inversible $\cL_K$ algébriquement équivalent à zéro, il existe un $\cO_{X}$-module inversible $\cL'$ algébriquement équivalent à zéro prolongeant $\cL_{K}^{\otimes n}$. Autrement dit, le conoyau de la flèche de restriction
$$\Pic^0(X/S)\ra\Pic^0(X_K)$$
est tué par $n$.
\end{Cor}

\begin{proof}
Soit $a_K$ l'image de $\cL_K$ dans $\Pic_{X_K/K}^0(K)=A_K(K)$. Le point $n\cdot a_K$ s'étend en une section $a'\in A^0(S)$. D'après la proposition \ref{surj}, on peut relever $a'$ en $b'\in\big((P,Y)'\big)^0(S)$. Maintenant, le morphisme de lissification
$(P,Y)'\ra (P,Y)$ induit une inclusion
$$\big((P,Y)'\big)^0(S)\subseteq (P,Y)^0(S).$$
La section $b'$ correspond via cette inclusion à la classe d'un couple $(\cL',\alpha)$. Comme la flèche composée
$$(P,Y)^0\ra (P,Y) \ra P$$
se factorise par $P^0$, qui est un sous-foncteur de $\Pic_{X/S}^0$, le $\cO_{X}$-module inversible $\cL'$ est algébriquement équivalent à zéro.
 
D'autre part, le diagramme commutatif (cf. remarque \ref{com})
\begin{displaymath}
\xymatrix{
(P,Y)' \ar[rr] \ar[d] & & A \ar[d] \\
(P,Y) \ar[r] & P \ar[r] & Q
}
\end{displaymath}
montre que le $\cO_{X_K}$-module inversible $\cL'\otimes K$ représente $a_{K}'=n\cdot a_K$. Comme $$\Pic(X_K)\ra\Pic_{X_K/K}(K)$$ est injectif (et même bijectif, cf. lemme \ref{extension}), le faisceau $\cL'\otimes K$ est donc isomorphe à $\cL_{K}^{\otimes n}$.
\end{proof}

\end{document}